\documentclass[10pt]{amsart}

\usepackage{amsmath} \usepackage{amssymb, amscd,cancel, graphicx,soul}
	\usepackage{pinlabel,mathtools}

\usepackage{hyperref}
\hypersetup{
    colorlinks=true, 
    linktoc=all,     
    linkcolor=black,  
}

\usepackage{enumerate}

\usepackage{cleveref}

\usepackage{tikz-cd}

\usepackage{mathdots}
\newtheorem{thm}{Theorem}[section]
\newtheorem{letteredthm}{Theorem}

\newtheorem{cor}[thm]{Corollary}
\newtheorem*{cor*}{Corollary}
\newtheorem{lem}[thm]{Lemma}
\newtheorem{prop}[thm]{Proposition}

\newtheorem*{ques*}{Question}
\newtheorem{defin}[thm]{Definition}

\newtheorem*{example*}{Example}
\newtheorem*{porism*}{Porism}
\newtheorem*{scholium*}{Scholium}

\newtheorem*{thm*}{Theorem}
\newtheorem*{defin*}{Definition}
\newtheorem*{lem*}{Lemma}
\newtheorem*{prop*}{Proposition}
\newtheorem*{remark*}{Remark}

\def\cA{{\mathcal A}}

\def\cH{{\mathcal H}}
\def\cJ{{\mathcal J}}
\def\cM{{\mathcal M}}

\def\w{{\mathrm{Width}}}
\def\JF{{\rm JF}}
\def\JFC{{\rm JFC}}
\def\CJF{{\rm JFC}}

\def\Jstd{{J_{\rm std}}}
\def\pid{{\pi_{\rm d}}}
\def\wt{\widetilde}

\def\bC{{\mathbb C}}
\def\C{{\mathbb C}}

\def\bF{{\mathbb F}}
\def\cG{{\mathcal G}}

\def\rad{{\mathrm{Rad}}}

\def\M{{\mathcal M}}

\def\bR{{\mathbb R}}

\def\bZ{{\mathbb Z}}

\def\area{{\mathrm{Area}}}

\def\jreg{{\mathcal{J}_{\mathrm{reg}}}}

\def\del{{\partial}}

\def\im{{\textup{im}}}

\newcommand{\ol}{\overline}

\begin{document}
\thispagestyle{empty}
\title[Floer homology and square pegs]{Floer homology and square pegs}
\author{Joshua Evan Greene} 
\address{Department of Mathematics, Boston College, USA}
\email{joshua.greene@bc.edu}
\urladdr{https://sites.google.com/bc.edu/joshua-e-greene}
\author{Andrew Lobb} 
\address{Mathematical Sciences,
	Durham University,
	UK}
\email{andrew.lobb@durham.ac.uk}
\urladdr{http://www.maths.dur.ac.uk/users/andrew.lobb/}
\thanks{JEG was supported by the National Science Foundation under Award No.~DMS-2304856 and by a Simons Fellowship.}

\begin{abstract}
We construct a version of Lagrangian Floer homology whose chain complex is generated by the inscriptions of a rectangle into a real analytic Jordan curve.
By using its associated spectral invariants, we establish that a rectifiable Jordan curve admits inscriptions of a whole interval of rectangles.
In particular, it inscribes a square if the area it encloses is more than half that of a circle of equal diameter.
\end{abstract}

\maketitle

\section{Introduction.}
\label{sec:intro}
This paper is motivated by the Square Peg Problem:
\begin{center}
{\em Does every Jordan curve in the plane inscribe a square?}
\end{center}

\noindent
Here a curve $\gamma$ {\em inscribes} a polygon $Q$ -- and $Q$ {\em inscribes in} $\gamma$ -- if $\gamma$ contains the vertices of an orientation-preserving similar copy of $Q$.
The problem was posed by Otto Toeplitz in 1911 \cite{toeplitz1911}.
It was affirmatively solved for {\em smooth} curves by Schnirelmann in 1929 \cite{schnirelman1929}, and it is known to be true for several other classes of curves \cite{matschke2014}.
However, the general case of {\em continuous} curves remains open to this day.
The difficulty with promoting the solution for smooth curves to continuous curves is the issue of {\em shrinkout}.

\subsection{Shrinkout and symplectic geometry.}
To describe the issue, suppose that $\gamma$ is a (continuous) Jordan curve and $\gamma_n$ is a sequence of smooth Jordan curves approximating it.
Schnirelmann's result shows that each $\gamma_n$ contains the vertices of a square $\square_n$.
It is tempting to use compactness to pass to a subsequence of $\square_n$ and conclude that they limit to a square $\square$ whose vertices are contained in $\gamma$.
However, the subsequence may limit to a single point $\bullet \in \gamma$.
Surmounting shrinkout has impeded progress on the Square Peg Problem for nearly a century.

The goal of this paper is to show how to preclude shrinkout for a wide class of curves using tools from symplectic geometry.
In order to describe the framework, we recall our earlier result specific to the case of smooth curves:
\begin{thm}
[\cite{greenelobb2}]
\label{cqpp}
Every smooth Jordan curve in the plane inscribes every cyclic quadrilateral.  In other words, every quadrilateral which inscribes in a circle also inscribes in every smooth Jordan curve.
\end{thm}
\noindent
The proof of \Cref{cqpp} proceeds by constructing a pair of embedded Lagrangian tori $L_0$ and $L_1$ in standard symplectic $\bC^2$ associated with a cyclic quadrilateral $Q$ and a smooth Jordan curve $\gamma$.
The intersection between $L_0$ and $L_1$ consists of a clean loop $C$ and a disjoint set $P$.
The set $P$ parametrizes the inscriptions of $Q$ in $\gamma$.
The loop $C$ parametrizes {\em degenerate} inscriptions, regarding each point on $\gamma$ as an degenerate copy of $Q$.
By surgering $C$ away, a single {\em immersed} Lagrangian torus $L$ results, and it self-intersects precisely in the set $P$.
This torus has minimum Maslov number 4.
On the other hand, an embedded Lagrangian torus in $\bC^2$ has minimum Maslov number 2.
(This is the resolution of Audin's conjecture for tori in $\bC^2$, independently due to Polterovich \cite{polto1991} and Viterbo \cite{viterbo1990}.) 
This shows that $L$ self-intersects, so $P$ is non-empty, and this implies the existence of the desired inscription of $Q$ in $\gamma$.

\Cref{cqpp} is optimal in a couple of senses.
Firstly, certainly no other quadrilateral can inscribe in every smooth Jordan curve since the circle is itself such a curve.
Secondly, the regularity condition on the curves cannot be weakened from smooth all the way down to continuous.
Indeed, for every cyclic quadrilateral which is not an isosceles trapezoid, there exists a continuous Jordan curve (even a triangle) which does not inscribe it.  This underscores the seriousness of shrinkout.
For suppose that $Q$ is a cyclic quadrilateral which is not an isosceles trapezoid, $\gamma$ is a triangle which does not inscribe $Q$, and $\gamma_n$ is a sequence of smooth approximations to $\gamma$ with $\gamma_n \rightarrow \gamma$ as $n \rightarrow \infty$.
\Cref{cqpp} shows that there exists a copy $Q_n$ of $Q$ inscribed in each $\gamma_n$.
However, they necessarily limit to a single point~$\bullet \in \gamma$.

\subsection{Jordan Floer homology.}
Nevertheless, the framework for proving \Cref{cqpp} contains a clue to precluding shrinkout for squares and, more generally, for rectangles.
The reason is that when $Q$ is a rectangle, the associated tori $L_0$ and $L_1$ are {\em monotone} and {\em Hamiltonian isotopic} to one another.
This suggests defining a version of Lagrangian Floer homology for the pair $(L_0,L_1)$ in order to gain greater control over inscribed rectangles in $\gamma$.
A large part of this paper carries out such a construction.

In sketch, the construction begins with a chain complex $\JFC(\gamma,\theta)$ associated with a generic pair of a {\em real analytic} Jordan curve $\gamma$ and an angle $0 < \theta < \pi$.
The restriction to this class of curves is made to overcome an analytical issue, but such curves are $C^0$-dense in the space of all Jordan curves.
The complex is generated by the set $P$ of inscriptions of a {\em $\theta$-rectangle} $Q_\theta$ in $\gamma$ (a $\theta$-rectangle is one whose diagonals meet in the angle $\theta$).
The differential $\del$ on the complex counts pseudoholomorphic strips with boundary on $L_0$ and $L_1$ which join pairs of points in $P \subset L_0 \cap L_1$.
A novelty in the construction, and a key part of the proof that $\del^2 = 0$, is that the strips must avoid a complex line which cuts through both $L_0$ and $L_1$ in the loop $C$.

The construction furnishes chain homotopy equivalences $\JFC(\gamma,\theta_1) \simeq \JFC(\gamma,\theta_2)$ for different choices of angles $0 < \theta_1,\theta_2 < \pi$, and for $\theta$ close to $0$ we identify $\JFC(\gamma,\theta)$ with a relative Morse chain complex of the pair $(L_0,C)$.
These properties lead to the fact that the total homology $\JF(\gamma,\theta)$ is two-dimensional.
Associated to each homology generator of $\JF(\gamma,\theta)$ is a {\em spectral invariant}, a real number induced from a filtration on the complex $\JFC$ by {\em action}.
The spectral invariants are robust under approximation of an arbitrary Jordan curve by real analytic ones.
This robustness provides control over shrinkout for rectangles for a wide class of curves.

\subsection{Square pegs.}
The application we give is the following result.
Recall that a Jordan curve $\gamma \subset \bR^2$ of finite length is {\em rectifiable}; in other words, $\gamma$ is rectifiable if it can be approximated by a sequence of smooth curves of uniformly bounded length.
Let $\area(\gamma)$ denote the area of the region that $\gamma$ encloses, and let $\rad(\gamma)$ denote half of its diameter, $\rad(\gamma) = \frac12 \sup \{ |z-w|  : z,w \in \gamma \}$.
Note that $\area(\gamma) / \rad(\gamma)^2 \le \pi$, with equality exactly when $\gamma$ is a circle.

\begin{letteredthm}
\label{thm:chunky}
If $\gamma \subset \bR^2$ is a rectifiable Jordan curve, then there exists an interval $I \subset (0, \pi)$ of length at least  $\area(\gamma)/\rad(\gamma)^2$ such that $\gamma$ inscribes every $\theta$-rectangle with $\theta \in I$.
\end{letteredthm}
\noindent
The following corollary is immediate.
\begin{cor*}
	\label{cor:chunky_squares_exist}
	If $\gamma$ is rectifiable and encloses more than half the area of a circle of equal radius, then $\gamma$ inscribes a square. \qed
\end{cor*}

\subsection{Overview of the construction of Jordan Floer homology.}
\label{ss:overview}
The paper unfolds as follows.
In \Cref{sec:JFC}, we begin with a triple of
\begin{itemize}
\item
a real analytic Jordan curve $\gamma \subset \bR^2 = \bC$,
\item
an angle $\theta \in (0,\pi)$, and
\item
analytic data $\mathbb{D}$ consisting of an admissible Hamiltonian and almost-complex structure.
\end{itemize}
From it we construct a $\bZ$-graded, $\bR$-filtered, $\bF_2$-chain complex
\[
\JFC_*(\gamma,\theta,\mathbb{D}) {\rm ,}
\]
which we shall term the \emph{Jordan Floer} complex.  Generically, $\JFC_*(\gamma,\theta,\mathbb{D})$ is freely generated by the transverse points of intersection between a pair of cleanly intersecting, monotone Lagrangian tori:
\[
L_0 = \gamma \times \gamma \subset \bC^2 \,\,\, {\rm and} \,\,\, L_1 = R_\theta(L_0) \subset \bC^2 {\rm .}
\]
Here $R_\theta : \bC^2 \to \bC^2$ denotes rotation through angle $\theta$ about the diagonal complex line
\[
\Delta(\bC) := \{(z,z) : z \in \bC\} \subset \bC^2 {\rm .}
\]
It is also the time-$\theta$ flow generated by the Hamiltonian function
\[
H \colon \bC^2 \longrightarrow \bR \colon (z,w) \longmapsto \frac14 |z-w|^2 {\rm ,}
\]
which shows that $L_1$ is Hamiltonian isotopic to $L_0$.
The transverse points of intersection between $L_0$ and $L_1$ are in one-to-one correspondence with inscriptions of $Q_\theta$ in $\gamma$: 
\[
(z,w) \in L_0 \cap L_1 \iff (z,w), R_\theta(z,w) \textup{ span diagonals of an inscription of } Q_\theta \textup{ in } \gamma.
\]
The clean loop of intersection
\[
\Delta(\gamma) = \{ (z,z) : z \in \gamma\} \subset L_0 \cap L_1
\]
corresponds to degenerate inscriptions in which the four vertices of the rectangle degenerate to a single point on the curve.

We introduce a class of admissible Hamiltonians $h_t$ in \Cref{subsec:admissible_curves}.
Their function is to treat the case in which $L_0$ and $L_1$ do not intersect transversely away from $\Delta(\gamma)$.
We introduce the moduli spaces of strips that we study in \Cref{subsec:strips} and the class of admissible almost-complex structures $J_t$ we require in  \Cref{subsec:ac_structures}.
The key requirement on an admissible $J_t$ is that it agrees with the standard almost-complex structure $\Jstd$ in a neighborhood of $\Delta(\bC)$, so that $\Delta(\bC)$ is a $J_t$-holomorphic divisor.
We describe the relationship between the construction of $\JFC(\gamma,\theta,\mathbb{D})$ presented here with a formulation in terms of Hamiltonian trajectories which begin and end on $L_0$ in \Cref{subsec:reformulation} 

We define the differential $\del$ on $\JFC_*(\gamma,\theta,\mathbb{D})$ in \Cref{subsec:differential}.
It counts certain pseudoholomorphic strips $u : \bR \times [0,1] \to \bC^2$.
As usual, the strips
\begin{enumerate}
\item
obey the Cauchy-Riemann equation $\overline{\del}_{\mathbb{D}} u = 0$,
\item
satisfy the boundary conditions $u(\bR \times 0) \subset L_0$ and $u(\bR \times 1) \subset L_1$,
\item
have bounded energy, and
\item
have Maslov index 1.
\end{enumerate}
Crucially, we impose one more condition: 
\begin{enumerate}
\item[(5)]
the closure of the image $u(\bR \times [0,1])$ is disjoint from $\Delta(\bC)$.
\end{enumerate}

Proving that $\del^2=0$ is the main content of \Cref{subsec:differential}.
This fact does not seem to follow from any known constructions.
The main issue is to analyze what happens in a smooth, 1-parameter family of strips $u_r$ which satisfy (1)-(3) above and such that some strip satisfies the disjointness condition (5).
As we show, the disjointness condition is then inherited by all other strips $u_r$ and their Gromov limits.
This requires a careful argument, since the divisor $\Delta(\bC)$ cuts through both tori $L_0$ and $L_1$.
This is where we use the fact that $\gamma$ is real analytic and invoke the conditions on our analytic data.

\Cref{thm:jf} asserts that the homology group is two-dimensional, supported in homological (Maslov) gradings 1 and 2:
\[
\JF_*(\gamma,\theta,\mathbb{D}) = (\bF_2)_{(2)} \oplus (\bF_2)_{(1)}.
\]
This follows by examining the dependence of the complex on $\theta$ and $\mathbb{D}$.
For a different choice of angle $0 < \theta' < \pi$ and analytic data $\mathbb{D}'$, we modify the definition of the differential to define a filtered chain homotopy equivalence $\JFC_*(\gamma,\theta,\mathbb{D}) \overset{\sim}{\to} \JFC_*(\gamma,\theta',\mathbb{D}')$.
This is done over the course of \Cref{subsec:continuation}.
When $\theta$ is chosen close to 0 or $\pi$, we show that $\JFC_*(\gamma,\theta,\mathbb{D})$ is isomorphic to the relative Morse chain complex $CM_*(\gamma \times \gamma, \Delta(\gamma); H)$, whose homology is in turn isomorphic to $(\bF_2)_{(2)} \oplus (\bF_2)_{(1)}$.
This is carried out in \Cref{subsec:morse}.
Note that the group is $\bZ$-graded, instead of $\bZ_2$-graded, as one might expect for a monotone Lagrangian with minimum Maslov number two.
The enhancement arises from using diagonal-avoiding strips in the construction of the invariant.

\subsection{Spectral invariants.}
As remarked, $\JFC_*(\gamma,\theta,\mathbb{D})$ also comes with an $\bR$-valued {\em action filtration}.
The action filtration of a generator is treated concretely in \Cref{sec:actions_and_rectangles}.

The action filtration on $\JFC_*(\gamma,\theta,\mathbb{D})$ descends to a filtration on $\JF_*(\gamma,\theta,\mathbb{D})$ studied in \Cref{sec:spectral_invariants}.
Each homology class in the group $\JF_*(\gamma,\theta,\mathbb{D})$ possesses a spectral invariant, its induced filtration grading.
All of the information in our setting is contained in the spectral invariant $\ell(\gamma,\theta)$ attached to the top dimensional generator.
As the notation suggests, it is independent of the choice of analytic data $\mathbb{D}$.

The application to \Cref{thm:chunky} follows from properties of this spectral invariant collected in \Cref{prop:properties_of_spectral_invariants}.
More precisely, $\ell(\gamma,\theta)$ is a continuous, monotonic non-decreasing function of $\theta \in (0,\pi)$ which limits to 0  as $\theta \to 0$ and to $\area(\gamma)$ as $\theta \to \pi$.
Since it is monotonic, it is differentiable almost everywhere, and its derivative is bounded above by $\rad(\gamma)^2$.
These properties enable us to establish \Cref{thm:chunky} for a real analytic Jordan curve $\gamma$ with bounds on the spectral invariants of the inscriptions of $Q_\theta$ in $\gamma$ with $\theta \in I$.

If $\gamma$ is a rectifiable Jordan curve, then we may approximate it by a sequence of real analytic Jordan curves $\gamma_n$ of bounded length by the Riesz-Prasolov theorem, a precise form of the Riemann mapping theorem for a domain with rectifiable boundary \cite[Theorem 6.8]{pommerenke}.
If $\theta \in I$ and $Q_n$ is an inscription of a $\theta$-rectangle in $\gamma_n$, then the action bound on $Q_n$ and the length bound on $\gamma_n$ combine to show that the sequence $Q_n$ converges to a nondegenerate $\theta$-rectangle $Q$ inscribed in $\gamma$.

\subsection{Varying the curve.}
In this paper we study how $\JF(\gamma,\theta)$ (and in particular its spectral invariants) vary with $\theta$.  It is a natural idea to try to perform the same kind of analysis while varying rather the Jordan curve $\gamma$ by a Hamiltonian flow on $\bC$.  There are some difficulties with this idea, the first analytic and the others topological, which we summarize here.

The analytic difficulty is that it is tricky to define continuation maps giving chain homotopy equivalences between chain complexes $\JFC(\gamma,\theta)$ for varying $\gamma$.  The problem is to guarantee that diagonal-avoiding strips do not degenerate in $1$-parameter families to strips meeting the diagonal.  
In a sequel \cite{greenelobb4}, we overcome this analytic difficulty by abandoning continuation maps and turning to Floer's original approach for defining chain maps, the \emph{bifurcation} method \cite{floerlag}.
In the current paper we borrow from this method in Section \ref{sec:spectral_invariants} when studying the rate of change of spectral invariants.

The second difficulty is that it seems hard, even in this setting, to circumvent the rectifiability hypothesis.

The third difficulty is that one would like to control the rate of change of the action with varying $\gamma$ in a way that might, for example, lead to the resolution of the Square Peg Problem for all rectifiable curves.
Looking ahead to Figures \ref{fig:icecreamaction} and \ref{fig:doubleicecream} which depict the action associated to an inscribed rectangle, we warn the reader that this example shows the rectangle {\em elegantly} inscribed in the Jordan curve.
It seems unlikely that for a general real analytic Jordan curve one should expect the spectral invariant to be attained by the action of such a rectangle.  For the Jordan curve $\gamma$ and rectangle $Q$ in Figure \ref{fig:icecreamaction}, any local isotopy of $\gamma$, performed away from the vertices of $\gamma$, that adds $\epsilon$ to the area bounded by $\gamma$, will change the action of the rectangle by at most $\epsilon$.  On the other hand, given any $N \gg 0$, it is possible to draw a Jordan curve with an inscribed rectangle whose action changes, under a small local isotopy adding $\epsilon$ to the area bounded by the curve, by $N \epsilon$.  That such bad rectangles can exist is at the root of this third difficulty.
The sequel \cite{greenelobb4} examines a case in which {\em a priori} all inscriptions are elegant, leading to some inscription results which fall outside the scope of \Cref{thm:chunky}.

\section*{Acknowledgments.}
\label{sec:acknowledgments}
We thank several mathematicians for useful conversations and sparing us their time during the course of this work: Denis Auroux, Paul Biran, Martin Bridgeman, Cole Hugelmeyer, Michael Hutchings, Joe Johns, Dusa McDuff, Tom Mrowka, Dietmar Salamon, Felix Schm\"ashke, and Chris Wendl.
In particular, Michael Hutchings pointed out that we were working with spectral invariants without knowing the term for them, which opened up the literature for us.

Part of this research was performed in Fall 2022 while the authors were visiting the Simons Laufer Mathematical Sciences Institute, which is supported by the National Science Foundation (Grant No. DMS-1928930).
We also gratefully acknowledge the Simons Foundation for hosting our visit to the 2024 Simons Collaboration on New Structures in Low-Dimensional Topology Annual Meeting.

\newpage
\section{Jordan Floer homology.}
\label{sec:JFC}

This section defines the Jordan Floer chain complex $\JFC(\gamma,\theta,h_t,J_t)$ for an admissible quadruple of Jordan curve $\gamma$, angle $\theta$, Hamiltonian perturbation $h_t$, and almost-complex structure $J_t$.

The first major result we prove in this section is that the boundary operator on the complex is indeed a differential.  The proof of this involves checking what happens as moduli spaces of strips degenerate, and verifying that these degenerations take place away from the diagonal $\Delta(\bC)$.

The remaining results follow from considering degenerations of more elaborate families of strips.  The reasons that these degenerations also stay away from the diagonal are already present in the proof of the first result, so we spend more time on this than the following results.  These results include that there exist continuation maps defining chain maps $\JFC(\gamma,\theta_1,h^1_t,J^1_t) \rightarrow \JFC(\gamma,\theta_2,h^2_t,J^2_t)$ for pairs of admissible quadruples with fixed $\gamma$.  Furthermore these chain maps are compatible with one another in the sense that any composition of any two such chain maps
\[ \JFC(\gamma,\theta_1,h^1_t,J^1_t) \longrightarrow \JFC(\gamma,\theta_2,h^2_t,J^2_t) \longrightarrow \JFC(\gamma,\theta_3,h^3_t,J^3_t) \]
is chain homotopy equivalent to any such chain map
\[  \JFC(\gamma,\theta_1,h^1_t,J^1_t) \longrightarrow \JFC(\gamma,\theta_3,h^3_t,J^3_t) {\rm .} \]
It follows firstly that we get a well-defined Jordan Floer homology group $\JF(\gamma,\theta)$.  Secondly, we get compatible isomorphisms
\[ \JF(\gamma, \theta_1) \longrightarrow \JF(\gamma, \theta_2) \]
for any choices $0 < \theta_1, \theta_2 < \pi$.

\subsection{Hamiltonians.}
\label{subsec:hamiltonians}
We recall some notation from the Introduction.
Our main interest lies in a very simple Hamiltonian:
\[
H \colon \bC^2 \longrightarrow \bR \colon (z,w) \longmapsto \frac{\vert z - w \vert^2}{4} {\rm .}
\]
We write the time-$\theta$ flow of this Hamiltonian $\Phi^{\theta}_H$ as $R_\theta \colon \bC^2 \rightarrow \bC^2$; this is a rotation of $\bC^2$ by an angle of $\theta$ about the axis given by the diagonal $\Delta(\bC) = \{ (z,z) : z \in \bC \} \subset \bC^2$.

\begin{defin}
\label{defin:smoothed_radial_Hamiltonian}
Let $\beta \colon [0,1] \rightarrow [0,\infty)$ be a smooth function such that $\int_0^1 \beta = 1$ and $\beta(t) =0$ for $0 \leq t \leq 0.1$ and for $0.9 \leq t \leq 1$.

We define $H_t(z,w) = \beta(t)H(z,w)$ for $0 \leq t \leq 1$ and $(z,w) \in \bC^2$.
\end{defin}

The function $H_t$ is a time-dependent Hamiltonian version of the Hamiltonian $H$, useful to us since for technical reasons we will wish our Hamiltonians to have zero derivative near $t=0,1$.  Again notice that the time-$\theta$ flow of $H_t$ is just $\Phi^\theta_{H_t} = R_\theta$.  In fact, we shall generally prefer flowing the Hamiltonian $\theta H_t$ for time $1$, which again gives $\Phi^1_{\theta H_t} = \Phi^\theta_{H_t} = R_\theta$.

\subsection{The chain group of an admissible triple.}
\label{subsec:admissible_curves}
Now let $\gamma \subset \bR^2 = \bC$ denote a {\em smooth} Jordan curve.
As such, it is Lagrangian with respect to the standard symplectic form on $\bC$, hence so is $\gamma \times \gamma \subset (\bC^2,\omega_{\mathrm{std}})$.
We would like, if we could, to study the Lagrangian Floer homology of the pair of Lagrangians
\[
L_0 = \gamma \times \gamma, \quad L_1 = R_\theta (\gamma \times \gamma).
\]
The first problem is that these Lagrangians have a clean loop of intersection $\Delta(\gamma)$,
and we plan to avoid this loop by analytic arguments which will involve restricting $\gamma$ to being {\em real analytic}.
The second problem is one that Lagrangian Floer theorists have well-developed techniques to handle -- that the remaining intersection points of $\gamma \times \gamma$ and $R_\theta (\gamma \times \gamma)$ may not be transverse.  The usual way to deal with this difficulty is by introducing a small Hamiltonian perturbation $h_t$ to achieve transversality.  We first introduce a distance which will be useful in choosing how close to the diagonal $\Delta(\bC)$ we take the small perturbations.

\begin{defin}
[Width]
	\label{defin:width}
	The {\em width} of a real analytic Jordan curve $\gamma$ is the infimal diameter $\w(\gamma)$ of the inscribed rectangles in $\gamma$.
\end{defin}
\noindent
Since the curvature of a smooth Jordan curve is uniformly bounded above (by compactness), a short argument (again by compactness) shows that width is positive: $\w(\gamma) > 0$.

\begin{defin}
[Admissible triple]
	\label{defin:admissible_triple}
	An \emph{admissible triple} $(\gamma, \theta, h_t)$ is a real analytic Jordan curve $\gamma$, an angle $0 < \theta < \pi$, and a time-dependent Hamiltonian perturbation $h_t \colon \bC^2 \rightarrow \bR$ for times $0 \leq t \leq 1$ satisfying the following conditions.
	\begin{enumerate}
		\item We have $h_t = 0$ whenever $0 \leq t \leq 0.1$ or $0.9 \leq t \leq 1$.
		\item We have $h_t(z,w) = 0$ whenever $\vert z - w \vert \leq \w(\gamma)/2$.
		\item The Lagrangians $\gamma \times \gamma$ and $\Phi^1_{\theta H_t + h_t} (\gamma \times \gamma)$ intersect cleanly along $\Delta(\gamma)$
		and everywhere else transversely.
	\end{enumerate}
\end{defin}
Notice that condition (2) ensures that $\Phi^1_{\theta H_t + h_t} (\gamma \times \gamma)$ coincides with $R_\theta(\gamma \times \gamma)$ in a neighborhood of $\Delta(\gamma)$.
Therefore, the condition that it intersects $\gamma \times \gamma$ cleanly along $\Delta(\gamma)$ actually follows from condition (2).
Since intersection points $\gamma \times \gamma \cap R_\theta(\gamma \times \gamma)$ are either in $\Delta(\gamma)$ or are at a distance of at least $\w(\gamma)$ from $\Delta(\bC)$, every pair $(\theta, \gamma)$ admits admissible $h_t$.

We write $X_{\cH_t}$ for the time-dependent Hamiltonian vector field associated to a Hamiltonian $\cH_t$.

\begin{defin}
[Trajectory]
	\label{defin:traj}
	A \emph{trajectory} $\tau \colon [0,1] \rightarrow \bC$ of $\theta H_t + h_t$ is an integral curve of the vector field $X_{\theta H_t + h_t}$
	\[
	\tau'(t) = X_{\theta H_t + h_t} \circ \tau(t)
	\]
	such that $\tau(0), \tau(1) \in \gamma \times \gamma$.
\end{defin}

Trajectories thus correspond to intersection points of $\gamma \times \gamma$ and $\Phi_{\theta H_t + h_t}(\gamma \times \gamma)$.
In particular, in the unperturbed case $h_t = 0$, these are points of $(\gamma \times \gamma) \cap R_{\theta} (\gamma \times \gamma)$.

In Lagrangian Floer theory, one typically works with generators drawn from
the space of pairs $(\tau, [\widehat{\tau}])$, where $\widehat{\tau} : [0,1] \times [0,1] \to \bC^2$ is a {\em capping} of $\tau$.  A capping is a smooth map obeying the conditions
\[
\widehat{\tau}(0,t) = \tau(t), \quad \widehat{\tau}(s,0) = \tau(0), \quad \widehat{\tau}(s,1) = \tau(1), \quad \widehat{\tau}(1,t) \in \gamma \times \gamma, \quad \forall s,t \in [0,1],
\]
and $[\widehat{\tau}]$ denotes its homotopy class relative to these conditions.
The construction of the Jordan Floer homology chain complex is smaller, due to the following result.

\begin{lem}
	[Preferred capping]
	\label{lem:capping}
	Each trajectory $\tau$ which is disjoint from the diagonal possesses a unique homotopy class of capping $[\widehat{\tau}]$ such that $\widehat{\tau}$ is disjoint from the diagonal: $\im(\widehat{\tau}) \cap \Delta(\bC) = \emptyset$.
\end{lem}

\begin{proof}
	This is clear once one makes the observation that the core curve of the cylinder $\gamma \times \gamma \setminus \Delta(\gamma)$ has winding number $1$ around $\Delta(\bC) \subset \bC^2$.
\end{proof}

\begin{defin}[Generators of the Jordan Floer chain complex]
	\label{defin:generators}
	Given an admissible triple $(\gamma,\theta,h_t)$, let $\cG(\gamma,\theta,h_t)$ denote the set of non-constant trajectories of $\theta H_t + h_t$:
	\begin{equation}
	\label{e:trajectory}
	\tau : [0,1] \to \bC^2, \quad \tau(0), \tau(1) \in \gamma \times \gamma, \quad {\tau}'(t) = X_{\theta H_t + h_t} \circ \tau(t).
	\end{equation}
	Let $\JFC(\gamma,\theta)$ denote the $\bF_2$-vector space freely generated by $\cG(\gamma,\theta,h_t)$.
\end{defin}

Thus if $(\gamma, \theta, 0)$ is admissible we may identify $\cG(\gamma, \theta, 0)$ with the set
\[ (\gamma \times \gamma \cap R_\theta(\gamma \times \gamma))\setminus \Delta(\gamma) {\rm .} \]
A trajectory obeying \eqref{e:trajectory} is constant if and only if it is contained in the diagonal.
It follows that each $\tau \in \cG(\gamma,\theta)$ has a preferred homotopy class of capping.

\begin{lem}
[Trajectories stay away from the diagonal]
\label{lem:distant_trajectory}
Each trajectory $\tau \in \cG(\gamma,\theta,h_t)$ remains at distance $> \w(\gamma)/2$ from $\Delta(\bC)$: if $\tau(t) = (z,w)$, then $|z-w| > \w(\gamma)/2$.
\end{lem}

\begin{proof}
The flow of $X_{\theta H_t + h_t}$ preserves the level sets of $|z-w|$ in the neighborhood
\[
N = \{(z,w) : |z-w| \le \w(\gamma)/2\},
\]
since the flow of $X_{\theta H_t}$ does so and $h_t$ vanishes on this set.
Hence any trajectory of $X_{\theta H_t + h_t}$ which enters $N$ is in fact contained in $N$ and is thus a trajectory of $X_{\theta H_t}$.
On the other hand, a nonconstant trajectory of $X_{\theta H_t}$ begins at a point $(z,w)$ which consists of the endpoints of a diagonal of an inscribed $\theta$-rectangle, and therefore obeys $|z-w| \ge \w(\gamma)$.
Hence no nonconstant trajectory of $X_{\theta H_t + h_t}$ can enter $N$.
\end{proof}

The importance of \Cref{lem:distant_trajectory} becomes apparent in \Cref{subsec:ac_structures} and beyond.
All trajectories in $\cG(\gamma,\theta,h_t)$ lie entirely outside the neighborhood in which the almost-complex structures we use are prescribed to be standard.
This is important in achieving transversality for the moduli spaces used in defining the Lagrangian Floer package of differentials, chain maps, and chain homotopy equivalences.

\subsection{Strips.}
\label{subsec:strips}

The Jordan Floer chain group $\JFC(\gamma, \theta, h_t)$ depends on a triple of data.
The Floer differential is defined in terms of a count of `strips' satisfying a Cauchy-Riemann-Floer equation, and shall therefore moreover require a choice of almost-complex structure.

Let $\cJ(\bC^2,\omega)$ denote the space of smooth, almost-complex structures on $\bC^2$ which are {\em compatible} with $\omega$ in the usual sense that
\[
|\cdot|_{J} := \omega( \cdot, J \cdot)
\]
defines a positive definite quadratic form on $T \bC^2$.
Distinguished amongst these is the {\em standard} almost-complex structure $\Jstd$ induced from multiplication by $\sqrt{-1}$.
All almost-complex structures on $\bC^2$ to follow are understood to be smooth and $\omega$-compatible.

Suppose that $(\gamma,\theta,h_t)$ is an admissible triple and $J_t$ is a time-dependent almost-complex structure.
We define the strip
\[ \Sigma \: = \bR \times [0,1] {\rm .} \]
We shall typically refer to points of $\Sigma$ by coordinates $(s,t) \in \Sigma$.
We consider smooth maps
\[
u \colon \Sigma \longrightarrow \bC^2
\]
which satisfy the boundary conditions
\begin{equation}
\tag{BC}
\begin{cases}
u(\bR \times \{0\}) \subset \gamma \times \gamma, \\[.2cm] u(\bR \times \{1\}) \subset \gamma \times \gamma;
\end{cases}
\end{equation}
which satisfy the non-linear Cauchy-Riemann-Floer equation
\begin{equation}
\label{e:CRF}
\tag{CRF}
\del_s u + J_t (\del_t u -X_{\theta H_t + h_t} \circ u)= 0;
\end{equation}
and which have bounded energy
\begin{equation}
\tag{BE}
\int |\del_s u|_{J_t}^2 ds \, dt < \infty.
\end{equation}
Subject to (BC) and (CRF), the condition (BE) is equivalent to the condition that the limits
\begin{equation}
\tag{LIM}
u(\pm \infty,t) := \lim_{s \to \pm \infty} u(s,t)
\end{equation}
exist (with exponential convergence in $s$) and define Hamiltonian trajectories of $\theta H_t + h_t$ in the sense of \eqref{e:trajectory}.
Note that either trajectory $u(-\infty,t), u(\infty,t)$ could be constant.
It follows in any case that $u$ extends uniquely to a continuous function on the extended strip $\overline{\Sigma} = [-\infty,\infty] \times [0,1]$, where $[-\infty,\infty]$ is topologized as the end compactification of $\bR = (-\infty,\infty)$.

\begin{defin}
	[Moduli spaces of strips]
	\label{defin:moduli_space_strips}
We define the moduli space of strips
\[
\cM(\gamma,\theta,h_t,J_t) := \{ u \in C^\infty(\Sigma,\bC^2) : \textup{(BC), (CRF), (BE)} \},
\]
the \emph{restricted} subspace of those strips whose which limit to non-constant trajectories
\[
\cM^\circ(\gamma,\theta,h_t,J_t) = \{ u \in \cM(\gamma,\theta,h_t,J_t) : u(\pm \infty, t) \in \cG(\gamma,\theta,h_t) \},
\]
and the \emph{diagonal-avoiding} subspace of those strips whose closures are disjoint from the diagonal
\[
\cM^\Delta(\gamma,\theta,h_t,J_t) = \{ u \in \cM(\gamma,\theta,h_t,J_t) : u({\overline{\Sigma}}) \cap \Delta(\bC) = \emptyset \}.
\]
\end{defin}
\noindent Since all constant trajectories lie in the diagonal $\Delta(\gamma) \subset \Delta(\bC)$, we have
\[ \cM^\Delta(\gamma,\theta,h_t,J_t) \subseteq \cM^\circ(\gamma,\theta,h_t,J_t) \subseteq \cM(\gamma,\theta,h_t,J_t) {\rm .} \]

\subsection{Admissible quadruples and the differential.}
\label{subsec:ac_structures}
We now wish to complete an admissible triple $(\gamma,\theta,h_t)$ to a quadruple $(\gamma,\theta,h_t, J_t)$ by choosing a suitable almost-complex structure $J_t$.  The suitability of $J_t$ will entail $J_t$ meeting the requirements of transversality and furthermore being well-behaved (in fact, being standard) in a neighborhood of the diagonal $\Delta(\bC)$ and of the boundary of the strip $\partial \Sigma$.  The former requirements are a necessity for showing that the differential is well-defined and squares to zero and would be addressed in any reasonable introduction to Floer homology.  The latter good behavior is a feature of our situation which will allow us to keep our strips from degenerating at the clean intersection $\Delta(\bC)$.

\begin{prop}
	[Regularity]
	\label{p:baire1}
	Let $(\gamma,\theta,h_t)$ be an admissible triple.
	There exists a Baire subset
	\[
	\jreg(\gamma,\theta,h_t) \subset C^\infty([0,1],\cJ(\bC^2,\omega))
	\]
	of time-dependent almost-complex structures 
	with the following properties for all $J_t \in \jreg(\gamma,\theta,h_t)$.
	\begin{enumerate}
		\item
		The restricted moduli space $\cM^\circ(\gamma,\theta,h_t,J_t)$ is a smooth manifold.
		\item In each dimension,
		$\cM^\circ(\gamma,\theta,h_t,J_t)$ contains finitely many components.
		\item
		The component of $\cM^\circ(\gamma,\theta,h_t,J_t)$ containing the map $u$ has dimension equal to the Maslov index $\mu(u) \in \bZ$.
		\item Finally, $J_t{(z,w)} = \Jstd \, \textup{ if } 0 \le t \le 0.1, \; \;  0.9 \le t \le 1, \textup{ or } \vert z - w \vert \leq \w(\gamma)/2$.
	\end{enumerate}
\end{prop}

\begin{proof}
Without condition (4), the result is standard: see Auroux's survey \cite[Section 1.4]{aurouxintro} for a quick discussion,
McDuff and Salamon's volume \cite[Sections 3.1 and 3.3]{slimmcduffsalamon} for an account of related results for closed holomorphic curves, 
and Schm\"aschke's thesis \cite{schmaschke} for a thorough treatment.
With condition (4) in play, we must additionally check that there is an open set $V_{\bC^2} \subset \bC^2$ and an open subset $V_\Sigma \subset \Sigma$ such that $J_t$ is unconstrained at points of $V_\Sigma$ when they are mapped into $V_{\bC^2}$.  Furthermore, for any continuous strip $w \colon \Sigma \rightarrow \bC^2$ with Lagrangian boundary conditions and limiting at either end to non-constant trajectories, we must verify that $w$ necessarily maps an open subset of $V_\Sigma$ into $V_{\bC^2}$.  (We thank Dusa McDuff for elucidating this point.)
In our case  \Cref{lem:distant_trajectory} tells us that each point $w(\pm \infty,t)$ lies at distance greater than $\w(\gamma)/2$ away from $\Delta(\bC)$, and the same is therefore true of $w(s,t)$ for all $|s| > S$ for some $S \gg 0$.
Hence we may take $V_\Sigma = \bR \times (0.1, 0.9)$ and $V_{\bC^2} = \{(z,w) \in \bC^2 : |z-w| > \w(\gamma)/2$\}.
\end{proof}

\begin{defin}
[Admissible quadruple]
\label{defin:admissible_quadruple}
If $(\gamma,\theta,h_t)$ is an admissible triple and $J_t \in \jreg(\gamma,\theta,h_t)$ then we call
\[ (\gamma, \theta, h_t, J_t) \]
an \emph{admissible quadruple}.
\end{defin}
\noindent
We briefly remark on condition (4) for the admissibility of $J_t$, which gets applied in proving that $\del^2 = 0$ in \Cref{subsec:differential}.
The condition that $J_t$ is standard close to the boundary of  $\Sigma$ in the source is used to rule out disk bubbling in the Gromov boundary of $\cM^\Delta(\gamma,\theta,h_t,J_t)$.
The condition that $J_t$ is standard close to $\Delta(\bC)$ in the target is used to argue that a family of strips disjoint from $\Delta(\bC)$ cannot limit to a (possibly broken) strip which intersects $\Delta(\bC)$.

\subsection{A reformulation.}
\label{subsec:reformulation}

The moduli space $\cM(\gamma,\theta,h_t,J_t)$ admits a useful reformulation by a technique explained briefly in \cite[Remark 1.10]{aurouxintro} and in more detail in \cite[Chapter 8]{audindamian}.  The idea is that trajectories of a Hamiltonian vector field which start and end on a Lagrangian correspond to intersection points of that Lagrangian with its time-$1$ flow under the vector field.  If one also suitably flows the strips of the moduli space then one arrives at new strips with boundary components on each Lagrangian, limiting at their ends to intersection points.
This reformulation gets used in the proof of \Cref{lem:diagonal_breaking} and elsewhere.

Consider the pair of transformations
\begin{equation}
\label{e:transformation}
u'(s,t) = (\Phi^t_{\theta H_t + h_t})^{-1} \circ u(s,t), \quad J'_t = (d \Phi^t_{\theta H_t + h_t})^{-1} \circ J_t \circ (d \Phi^t_{\theta H_t + h_t})
\end{equation}
of the spaces $C^\infty({\Sigma},\bC^2)$ and $C^\infty([0,1],\cJ(\bC^2,\omega))$.

\begin{lem}
\label{lem:jtransform}
If $J_t \in \cJ(\gamma,\theta,h_t)$ then $J'_t(z,w)$ of \eqref{e:transformation} agrees with $\Jstd$ whenever $\vert z - w \vert < \w(\gamma)/2$ and whenever $0 \leq t \leq 0.1$ or $0.9 \leq t \leq 1$.
\end{lem}

\begin{proof}
Suppose that $J_t \in \cJ(\gamma,\theta,h_t)$.  Then $J_t(z,w) = \Jstd$ whenever $\vert z - w \vert < \w(z,w)$.
The transformation $\Phi^t_{\theta H_t + h_t}$ effects a rotation (through angle $\theta  \int_0^t \beta(r) dr$) of $\bC^2$ about $\Delta(\bC)$ for all $0 \le t \le 1$.
This implies that $\Phi^t_{\theta H_t + h_t}$ is an automorphism of ${\Jstd}$ and that it preserves $H$.
Applying the first of these facts to $0 \le t \le 0.1$ and $0.9 \le t \le 1$, it follows that 
\[
J'_t = (d \Phi^t_{\theta H_t + h_t})^{-1} \circ J_t \circ (d \Phi^t_{\theta H_t + h_t}) = (d \Phi^t_{\theta H_t + h_t})^{-1} \circ {\Jstd} \circ (d \Phi^t_{\theta H_t + h_t}) = {\Jstd}.
\]
Applying the second of these facts to a point $(z,w)$ with $H(z,w) \le \mathrm{w}(\gamma,\theta)/2$, we find that $J_t = {\Jstd}$ at the point $\Phi^t_{\theta H_t + h_t}(z,w)$.
Now a similar derivation to the last, paying attention to the points of $\bC^2$ where the endomorphisms are applied, shows that $J'_t = {\Jstd}$ at $(z,w)$.
\end{proof}

Fix $J_t \in \cJ(\gamma,\theta,h_t)$ and let $J'_t$ be the resulting almost-complex structure from \eqref{e:transformation}.  Now consider smooth maps
\[
u' \colon \Sigma \longrightarrow \bC^2
\]
which satisfy the boundary conditions
\begin{equation}
\label{e:BC'}
\tag{BC$'$}
\begin{cases}
u'(\bR \times \{0\}) \subset \gamma \times \gamma, \\[.2cm] u'(\bR \times \{1\}) \subset (\Phi^1_{\theta H_t + h_t})^{-1}(\gamma \times \gamma);
\end{cases}
\end{equation}
the non-linear Cauchy-Riemann equation
\begin{equation}
\label{e:CR'}
\tag{CR$'$}
\del_s u' + J'_t (\del_t u')= 0;
\end{equation}
and which have bounded energy
\begin{equation}
\label{e:BE'}
\tag{BE$'$}
\int |\del_s u'|_{J'_t}^2 ds \, dt < \infty.
\end{equation}
Similar to before, subject to conditions \eqref{e:BC'} and \eqref{e:CR'}, condition \eqref{e:BE'} is equivalent to the condition that the limits
\begin{equation}
\tag{LIM'}
u'(\pm \infty,t) := \lim_{s \to \pm \infty} u'(s,t)
\end{equation}
exist, are independent of $t$, and equal points of intersection between $\gamma \times \gamma$ and $(\Phi^1_{\theta H_t + h_t})^{-1}(\gamma \times \gamma)$.

Define the moduli spaces
\[
\cM'(\gamma,\theta,h_t,J'_t) := \{ u' \in C^\infty(\Sigma,\bC^2) : \textup{(BC$'$), (CR$'$), (BE$'$)} \}
\]
and
\[
(\cM')^\Delta(\gamma,\theta,h_t,J'_t) := \{ u' \in \cM'(\gamma,\theta,h_t,J'_t) : u'(\overline{\Sigma}) \cap \Delta = \emptyset \}.
\]
The proof of the following result follows that of \cite[Remark 1.10]{aurouxintro}.

\begin{prop}
The transformation \eqref{e:transformation} carries $\cM(\gamma,\theta,h_t,J_t)$ to $\cM'(\gamma,\theta,h_t,J'_t)$ and furthermore carries $\cM^\Delta (\gamma,\theta,h_t,J_t)$ to $(\cM')^\Delta(\gamma,\theta,h_t,J'_t)$. \qed
\end{prop}

\subsection{Proof that $\del$ is a differential.}
\label{subsec:differential}
We fix an admissible quadruple $(\gamma, \theta, h_t, J_t)$ of real analytic Jordan curve, angle, Hamiltonian perturbation, and almost-complex structure.  In this subsection we shall define a map
\[ \partial = \partial_{(\gamma, \theta, h_t, J_t)} \colon \JFC(\gamma, \theta, h_t) \longrightarrow \JFC(\gamma, \theta, h_t) \]
and verify that it is a differential: $\partial^2 = 0$.  We can then make the following definition.

\begin{defin}
	\label{defin:JF_from_JFC}
	We define the Jordan Floer homology
	\[ \JF(\gamma, \theta, h_t, J_t) := \frac{\ker{\partial_{(\gamma, \theta, h_t, J_t)}}}{\im{\partial_{(\gamma, \theta, h_t, J_t)}}} {\rm .}\]
\end{defin}
Later we shall see that there is no dependence on the choice of perturbation $h_t$ or admissible almost-complex structure~$J_t$, so that we can write rather $\JF(\gamma, \theta)$.

Let $\cM_k$ denote the union of the $k$-dimensional components of a manifold $\cM$.
All of the restricted and diagonal-avoiding moduli spaces $\cM$ defined above support free $\bR$-actions by reparametrization: $r \cdot u(s,t) = u(r+s,t)$ for $r \in \bR$ and $u(s,t) \in \cM$.
We let $\widehat{\cM} = \cM / \bR$ and let $\widehat{u} \in \widehat{\cM}$ denote the {\em unparametrized} strip corresponding to the $\bR$-equivalence class of $u \in \cM$.

\begin{defin}
[Differential on the Jordan Floer chain complex]
	\label{defin:Jordan_Floer_differential}
	Let $\tau \in \cG(\gamma, \theta, h_t)$ be a generator of the $\bF_2$-vector space $\CJF(\gamma, \theta, h_t)$ (see \Cref{defin:generators}).  We define the map \[\partial = \partial_{(\gamma, \theta, h_t, J_t)} \colon \JFC(\gamma, \theta, h_t) \longrightarrow \JFC(\gamma, \theta, h_t)\]
	by requiring that
	\[ \partial (\tau) = \sum_{\substack{\tau' \in \cG(\gamma, \theta, h_t) \\ \widehat{u} \in \widehat{\cM_1^\Delta}(\gamma, \theta, h_t, J_t) \\ u(-\infty,t) = \tau(t) \\ u(\infty,t) = \tau'(t)  }} \tau' {\rm .}\]
\end{defin}

We now confirm that we have a well-defined chain complex by verifying that $\partial^2 = 0$.

\begin{thm}
[$\del^2 =0$]
\label{thm:differential}
Suppose that $(\gamma,\theta,h_t,J_t)$ is an admissible quadruple.
Then the boundary operator $\partial = \partial_{(\gamma, \theta, h_t, J_t)}$ on $\JFC(\gamma,\theta,h_t)$ satisfies $\del^2 = 0$.
\end{thm}
\noindent To prove this theorem we lean on the usual technology of Floer homology.
The difficulties arise because we are limiting our attention to the diagonal-avoiding moduli spaces.
The proof of \Cref{thm:differential} consists of verifying that the $1$-dimensional reduced diagonal-avoiding moduli spaces can be compactified by adding in products of appropriate pairs of elements of the $0$-dimensional reduced diagonal-avoiding moduli spaces.  This will establish the theorem since we know that, for admissible quadruples $(\gamma, \theta, h_t, J_t)$, moduli spaces of strips limiting at each end to non-constant trajectories are manifolds of the expected dimension.

Schm{\"a}schke \cite{schmaschke} (see in particular Chapter 5) considers the case (of which ours is an example) of two monotone Lagrangians in clean intersection, and studies the possible degenerations of strips with Lagrangian boundary conditions satisfying the Cauchy-Riemann-Floer equation.  He uses this to establish that moduli spaces of strips in this setting admit Gromov compactifications.  Essentially the situation is very similar to that of Lagrangians in transverse intersection.  Schm{\"a}schke establishes (\cite[Theorem 5.1.4]{schmaschke} gives a precise statement) that moduli spaces of strips may limit to concatenations of strips and trees of bubbles.

Let $\cM \subseteq \widehat{\cM_2^\circ}(\gamma, \theta, h_t, J_t)$ be a 1-dimensional component of the moduli space of restricted unparametrized strips.
For $0 \leq r < 1$, we consider a possible path $\widehat{u_r}$ in this component represented by parametrized diagonal-avoiding strips $u_r \in \cM^\Delta_2(\gamma, \theta, h_t,J_t)$.  There is a well-defined limit ``strip" $\widehat{u_1}$ that could \emph{a priori} fall into one of several classes.

\begin{enumerate}
	\item {\bf Touching the diagonal.}  The limit $\widehat{u_1} \in \widehat{\cM^\circ_2}(\gamma, \theta, h_t,J_t) \setminus \widehat{\cM^\Delta_2}(\gamma, \theta, h_t, J_t)$.
	\item {\bf Breaking at the diagonal.}  The limit $\widehat{u_1}$ is the concatenation of two strips $\widehat{u_1} = \widehat{u_1^-} \# \widehat{u_1^+}$ in which $u_1^-$ and $u_1^+$ are strips satisfying $\lim_{s \rightarrow +\infty} u_1^-(s,t) = (p,p) = \lim_{s \rightarrow -\infty} u_1^+(s,t)$ for some $p \in \gamma$ and for all $0 \leq t \leq 1$.
	\item {\bf Sphere bubbling.}  The limit $\widehat{u_1}$ contains a sphere bubble.
	\item {\bf Disk bubbling.} The limit $\widehat{u_1}$ contains a disk bubble.
	\item {\bf Good breaking.}  The limit $\widehat{u_1} \in \widehat{\cM^\Delta_2}(\gamma, \theta, h_t, J_t) \cup (\widehat{\cM^\Delta_1}(\gamma, \theta, h_t,J_t) \times \widehat{\cM^\Delta_1}(\gamma, \theta, h_t,J_t))$.
\end{enumerate}

We would like to verify that the first cases do not occur, so that we necessarily land in the final case.  In the final case, if we have that $\widehat{u_1} \in \widehat{\cM^\Delta_2}(\gamma, \theta, h_t,J_t)$, this is verifying that a path of diagonal-avoiding Cauchy-Riemann-Floer strips cannot wander up to and touch the diagonal (hence leaving the space of diagonal-avoiding strips).  On the other hand, if we have that $u_1 \in \widehat{\cM^\Delta_1}(\gamma, \theta, h_t,J_t) \times \widehat{\cM^\Delta_1}(\gamma, \theta, h_t,J_t)$ this is verifying that the $1$-dimensional space $\widehat{\M^\Delta_2}(\gamma, \theta, h_t,J_t)$ can be compactified by adding suitable elements of $\widehat{\cM^\Delta_1}(\gamma, \theta, h_t,J_t) \times \widehat{\cM^\Delta_1}(\gamma, \theta, h_t,J_t)$.

\begin{proof}[Proof of Theorem \ref{thm:differential}.]
We establish Theorem \ref{thm:differential} by verifying Lemmas \ref{lem:touching_the_diagonal}, \ref{lem:diagonal_breaking}, \ref{lem:sphere}, and \ref{lem:disc_bubbling}, which exclude cases (1), (2), (3), (4) of the above discussion in turn.
\end{proof}

\begin{lem}
	[\bf Touching the diagonal]
	\label{lem:touching_the_diagonal}
	The limit $\widehat{u_1} \notin \widehat{\cM_2}(\gamma, \theta, h_t,J_t) \setminus \widehat{\cM^\Delta_2}(\gamma, \theta, h_t,J_t)$.
\end{lem}
\begin{proof}
	Suppose for a contradiction that we have $\widehat{u_1} \in \widehat{\cM_2}(\gamma, \theta, h_t,J_t) \setminus \widehat{\cM^\Delta_2}(\gamma, \theta, h_t,J_t)$.
	
	We write $\partial \ol{\Sigma} = (\bR \times \{0,1\}) \cup (\{ -\infty, +\infty \} \times [0,1])$, and we choose $u_1 \colon \Sigma \rightarrow \bC^2$, a parametrized representative of $\widehat{u_1}$.
	
	If we are in the case where $u_1 (\partial \ol{\Sigma}) \cap \Delta(\bC) = \emptyset$, then $u_1 (\partial \ol{\Sigma})$ is a loop that has zero winding number around $\Delta(\bC)$, since each $u_r(\partial \ol{\Sigma})$ has zero winding number for $0 \leq r < 1$.  The condition that $\widehat{u_1} \in \widehat{\cM_2}(\gamma, \theta, h_t,J_t) \setminus \widehat{\cM^\Delta_2}(\gamma, \theta, h_t,J_t)$ implies that there must be at least one point $p \in \Sigma \setminus \partial \Sigma$ such that $u_1 (p) \in \Delta(\bC)$.  If all such $p$ are isolated then this will contradict the positivity of such isolated intersection points.  Here we refer to \cite[Lemma 4.3]{ganatrapormerleano} which establishes positivity of intersection between holomorphic divisors preserved by Hamiltonian flow and strips satisfying the Hamiltonian-perturbed Cauchy-Riemann-Floer equation.  In our case, $J_t$ is $\Jstd$ near the complex line $\Delta(\bC)$, while the derivative of the Hamiltonian $\theta H_t + h_t$ vanishes on $\Delta(\bC)$.
	
	If, on the other hand, not every such $p$ is isolated then analytic continuation implies that $u_1(\Sigma) \subseteq \Delta$ which would give another contradiction.
	
	Suppose then that $u_1(s_0,t_0) = (p,p) \in \Delta(\gamma) \subset \Delta(\bC)$ for some $s_0 \in \bR$ and $t_0 \in \{ 0 , 1 \}$.  Choose parametrized representatives $u_r \colon \Sigma \rightarrow \bC^2$ for $0 \leq r < 1$ so that $r \mapsto u_r(s, t)$ is smooth for all $(s,t) \in \Sigma$.
	
	Each $u_r$ satisfies the Cauchy-Riemann-Floer equation and has boundary conditions on $\gamma \times \gamma$.  Near the boundaries of the strip $\Sigma$ we have chosen the Hamiltonian perturbation $h_t$ and the almost complex structure $J_t$ so that the Cauchy-Riemann-Floer equation reduces to the Cauchy-Riemann equation with respect to the standard complex structure $\Jstd$.
	
	More precisely, there exists an $\epsilon_0 > 0$ (indeed, $\epsilon_0 = 0.1$) such that the restrictions to a half-disc $u_r \vert_{\Sigma \cap B_{\bR^2}((s_0,t_0), \epsilon_0)}$ satisfy the Cauchy-Riemann equation with respect to $\Jstd$ for all $0 \leq r \leq 1$.  Recall now that $\gamma$ is real analytic, so we have that $\gamma \times \gamma$ is real analytic with respect to $\Jstd$.  The Schwarz reflection principle then tells us that there exists a small disc $B_{\bR^2}((s_0,t_0),\epsilon_1)$ where $0 < \epsilon_1 < \epsilon_0$ has been chosen uniformly by compactness such that $u_r \vert_{\Sigma \cap B_{\bR^2}((s_0,t_0), \epsilon_1)}$ admits a $\Jstd$-holomorphic extension to
	\[ \ol{u_r} \colon B_{\bR^2} ((s_0,t_0), \epsilon_1) \longrightarrow \C^2 {\rm .} \]
	
	Let $B \subset \bC^2$ denote a neighborhood of $(p,p)$ such that $J |_B = \Jstd$, which contains the images of all $\ol{u_r}$ when $r$ is close enough to $1$, and such that Schwarz reflection across $(\gamma \times \gamma) \cap B$ carries $B$ into itself.
	Since the diagonal $\Delta(\bC)$ is $\Jstd$-holomorphic, both $\Delta(\bC) \cap B$ and its image under the Schwarz reflection are $\Jstd$-holomorphic.
	Since they intersect along the codimension 1 curve $\Delta(\gamma)$, they coincide by analytic continuation.
	Thus, $\Delta(\bC) \cap B$ is taken into itself by Schwarz reflection.
	It follows that any point of intersection between $\ol{u_r}$ and $\Delta(\bC)$ is either a point of intersection between $u_r$ and $\Delta(\bC)$ or is the image of such a point under the reflection.
	However, $u_r$ is disjoint from $\Delta(\bC)$, so the same follows for $\ol{u_r}$.

	If $(p,p)$ is a non-isolated intersection point between $\ol{u_1}$ and $\Delta(\bC)$, then analytic continuation implies that the image of $\ol{u_1}$ is contained in $\Delta(\bC)$, and further continuation implies that the image of $u_1$ is contained in $\Delta(\bC)$, a contradiction.
	
	On the other hand, if $(p,p)$ is an isolated intersection point, then it must be a positive intersection point and in particular an intersection point of non-zero multiplicity.  Hence for $r$ close to $1$, $\ol{u_r}$ must also intersect $\Delta(\bC)$, contradicting that $u_r$ is disjoint from $\Delta(\bC)$.
\end{proof}

\begin{lem}
	[\bf Breaking at the diagonal]
	\label{lem:diagonal_breaking}
	We do not have diagonal breaking.
\end{lem}
\begin{proof}
	Suppose for a contradiction that $\widehat{u_1}$ is the concatenation of two strips $\widehat{u_1} = \widehat{u_1^-} \# \widehat{u_1^+}$ in which the parametrized strips $u_1^-$ and $u_1^+$ satisfy $\lim_{s \rightarrow +\infty} u_1^-(s,t) = (p,p) = \lim_{s \rightarrow -\infty} u_1^+(s,t)$ for some $p \in \gamma$ and for all $0 \leq t \leq 1$.

	First note that one can port the argument from the proof of Lemma~\ref{lem:touching_the_diagonal} to show that we cannot have $u_1^-(\Sigma)$ or $u_1^+(\Sigma)$ having non-empty intersection with $\Delta(\bC)$.
	
	We choose $R \gg 0$ so that $u_1^-([R,\infty) \times [0,1])$ and $u_1^+((-\infty,-R] \times [0,1])$ are close enough to the diagonal so that $J_t$ is $\Jstd$ for all points of $u_1^-([R,\infty) \times [0,1]) \cup u_1^+((-\infty,-R] \times [0,1])$.
	
	We work with the reparametrizations $v_1^+$, $v_1^-$ of $u_1^+$, $u_1^-$ respectively
\[
		v_1^{\pm}(s,t) = (\Phi^t_{\theta H_t + h_t})^{-1} \circ u_1^\pm(s,t) {\rm ,}
\]
	which we introduced in Subsection~\ref{subsec:reformulation}.
	It follows from \Cref{lem:jtransform} that
	$v_1^+ \vert_{s \leq R}$ and $v_1^- \vert_{s \geq R}$ satisfy the Cauchy-Riemann equation
	\[ \partial_s v_1^\pm + \Jstd \partial_t v_1^\pm = 0 \]
	with respect to the standard almost complex structure $\Jstd$.

	Our plan is to study the projections of the holomorphic curves $v_1^\pm$ under the map
	\begin{equation}
	\pid \colon \bC^2 \longrightarrow \C \colon (z,w) \longmapsto z-w {\rm ,}
	\end{equation}
	using `d' for `difference'.
	Note that this projection map is holomorphic with respect to $\Jstd$ and maps the diagonal $\Delta(\bC) \subset \bC^2$ to the origin $0 \in \bC$.  We are going to be concerned in particular with the arguments of points on the curves $v_1^\pm$ under this projection map.
	
	Choose a point $p \in \gamma$ and a non-zero $v \in T_p \gamma \subset T_p \bC = \bC$.
	Next let $A \subset \gamma$ be a small open arc with $p \in A$.  By taking $A$ to be small enough we may ensure that
		\[ \{ \arg(z - w) : z, w \in A, z \not= w \} \subset (\arg(v) - \epsilon, \arg(v) + \epsilon ) \cup (\arg(v) + \pi - \epsilon, \arg(v) + \pi + \epsilon) \subset \bR/2\pi\bZ = S^1 \]
	for any small $\epsilon > 0$ that we choose.  
	In fact, we shall now choose $\epsilon > 0$ to be small relative to $\theta$:
	\[ \epsilon < \theta/2  \,\, {\rm and} \,\, \epsilon < (\pi - \theta)/2 \]	
	and choose the small interval $A \subset \gamma$ accordingly.
	
	By increasing $R$ if necessary, we ensure that
	\[ v_1^-( [R,\infty) \times \{ 0 \} ) \cup  v_1^+( (-\infty, -R] \times \{ 0 \} ) \subset A \times A \] and
	\[v_1^-( [R,\infty) \times \{ 1 \} ) \cup v_1^+( (-\infty, -R] \times \{ 1 \} ) \subset R_\theta^{-1}(A \times A) = R_{-\theta} (A \times A) {\rm .} \]
	
	Write $W \subset \bC$ for the region of the complex plane consisting of all points whose arguments lie in the set \[ (\arg(v) - \epsilon, \arg(v) + \epsilon) \cup (\arg(v) + \pi - \epsilon, \arg(v) + \pi + \epsilon) {\rm .} \]
	
	Then we see that
	\[ \pid \circ v_1^-( [R,\infty) \times \{ 0 \} ) \cup  \pid \circ v_1^+( (-\infty, -R] \times \{ 0 \} ) \subset W \] and
	\[\pid \circ v_1^-( [R,\infty) \times \{ 1 \} ) \cup \pid \circ v_1^+( (-\infty, -R] \times \{ 1 \} ) \subset e^{-i \theta}W {\rm ,} \]
	but $W \cap e^{-i \theta}W = \{ 0 \}$.
	
	Note also that for $r \not= 1$ we have that $v_r$ does not meet the diagonal $\Delta(\bC)$.  Hence we see that $\pid \circ v_1^-( [R,\infty) \times \{ 0 \} ) \cup  \pid \circ v_1^+( (-\infty, -R] \times \{ 0 \} )$ lies entirely in one component of $W \setminus \{ 0 \}$ while $\pid \circ v_1^-( [R,\infty) \times \{ 1 \} ) \cup \pid \circ v_1^+( (-\infty, -R] \times \{ 1 \} )$ lies entirely in one component of $e^{-i \theta} W \setminus \{ 0 \}$.
	
	\begin{figure}
		\labellist
		\pinlabel {$a$} at 530 820
		\pinlabel {$W$} at 690 -50
		\pinlabel {$e^{-i \theta} W$} at -50 150
		\endlabellist
		\centering
		\includegraphics[scale=0.1]{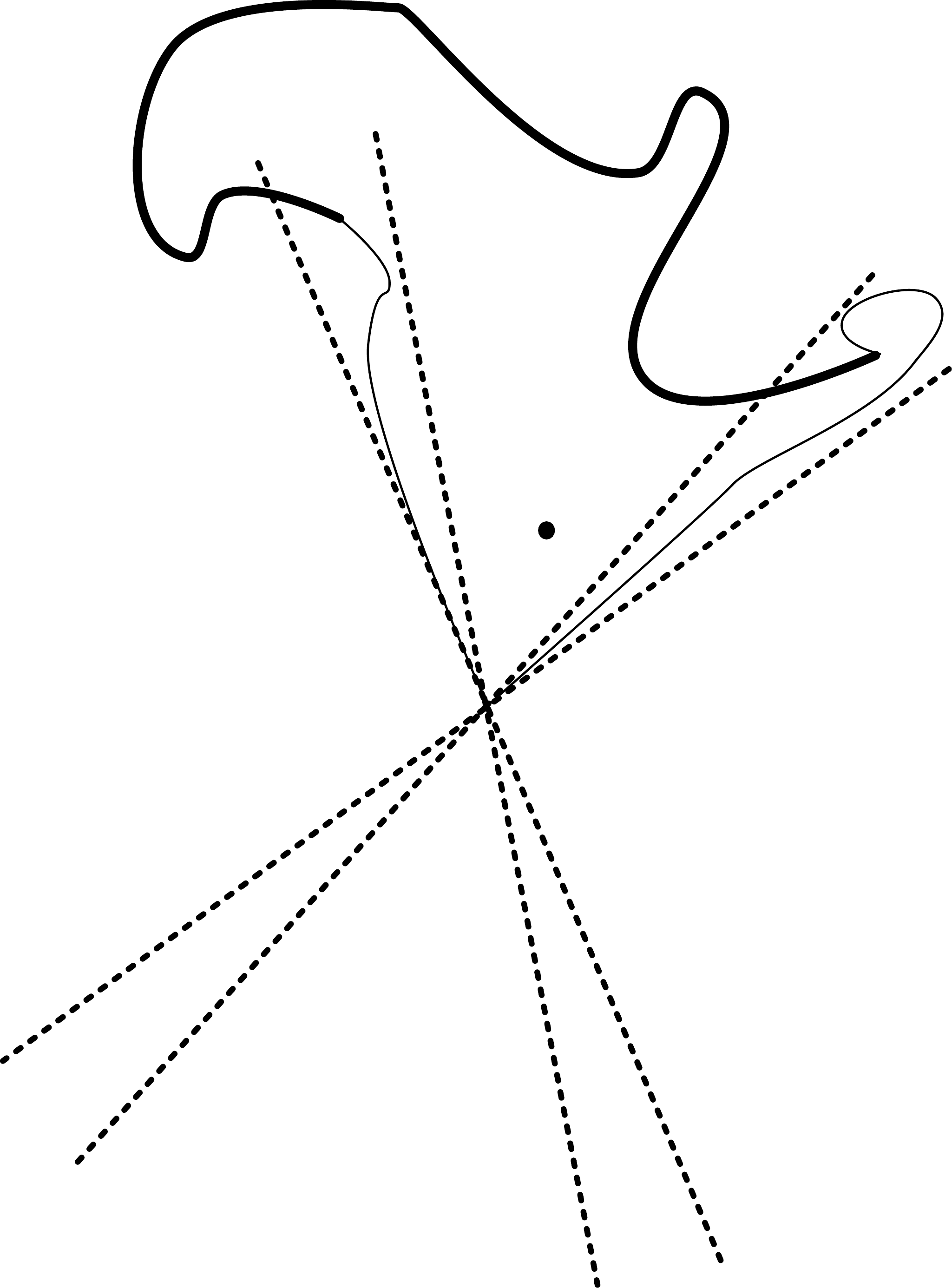}
		\caption{We show the two regions $W$ and $e^{-i \theta}W$ of the plane.  Within them we have drawn in bold $\pid \circ v_1^-(\{R\} \times [0,1])$ and with a finer nib $\pid \circ v_1^-([R, \infty) \times \{ 0,1 \})$.  By adding in $\{ 0 \} = \pid \circ v_1^-(\{\infty\} \times [0,1])$ we obtain a closed loop $\Lambda$ (possibly self-intersecting) in the plane.  The loop $\Lambda$ winds around the point in the plane labelled with an $a$.}
		\label{fig:breakingbad}
	\end{figure}

In Figure \ref{fig:breakingbad} we have indicated the closed loop $\Lambda$ in the plane given by
\[ \Lambda = \pid \circ v_1^-( \{ R , \infty \}  \times [0,1] \cup [ R, \infty ] \times \{ 0,1 \} ) {\rm.} \]
Letting $\delta>0$ stand for the minimum distance from the origin of the union \[ \pid \circ v_1^-( \{ R \}  \times [0,1] ) \cup \pid \circ v_1^+ (\{ -R \} \times [0,1] ) {\rm ,} \]
we choose a point $a \in \bC$ with $\vert a \vert = \delta/2$ and $a \notin (W \cup e^{-i \theta} W)$ such that the winding number of $\Lambda$ around $a$ is non-zero.  It follows in particular that we have
\[ a \in \pid \circ v_1^-((R,\infty) \times (0,1)) {\rm .} \]

	Finally, we observe that for all $r \in (0,1)$ sufficiently close to $1$ we may choose $\wt{R_r}, \wt{R_r}' \in \bR$ with $\wt{R_r} < \wt{R_r}'$ so that
	\begin{itemize}
		\item $v_r(\partial ([\wt{R_r}, \wt{R'_r}] \times [0,1]) ) \rightarrow v^-_1(\partial ([R,\infty]\times[0,1])) \cup v_1^+(\partial([-\infty,-R] \times [0,1]))$ as $r \rightarrow 1$,
		\item $\pid \circ v_r$ is holomorphic on $[\wt{R_r}, \wt{R_r}'] \times [0,1]$,
		\item $\pid \circ v_r( \{ \wt{R_r}, \wt{R_r}' \} \times [0,1])$ is distance at least $3\delta/4$ from the origin,
		\item $\pid \circ v_r( [\wt{R_r}, \wt{R_r}'] \times \{ 0 \} ) \subset ({\rm a} \,\, {\rm component} \,\, {\rm of} \,\, W \setminus \{0\})$,
		\item $\pid \circ v_r( [\wt{R_r}, \wt{R_r}'] \times \{ 1 \} ) \subset ({\rm a} \,\, {\rm component} \,\, {\rm of} \,\, e^{-i \theta}W \setminus \{0\})$, and
		\item $a \in \pid \circ v_r( (\wt{R_r}, \wt{R_r}') \times (0,1) )$.
	\end{itemize}

\begin{figure}
	\labellist
	\pinlabel {$a$} at 530 820
	\pinlabel {$W$} at 690 -50
	\pinlabel {$e^{-i \theta} W$} at -50 150
	\endlabellist
	\centering
	\includegraphics[scale=0.1]{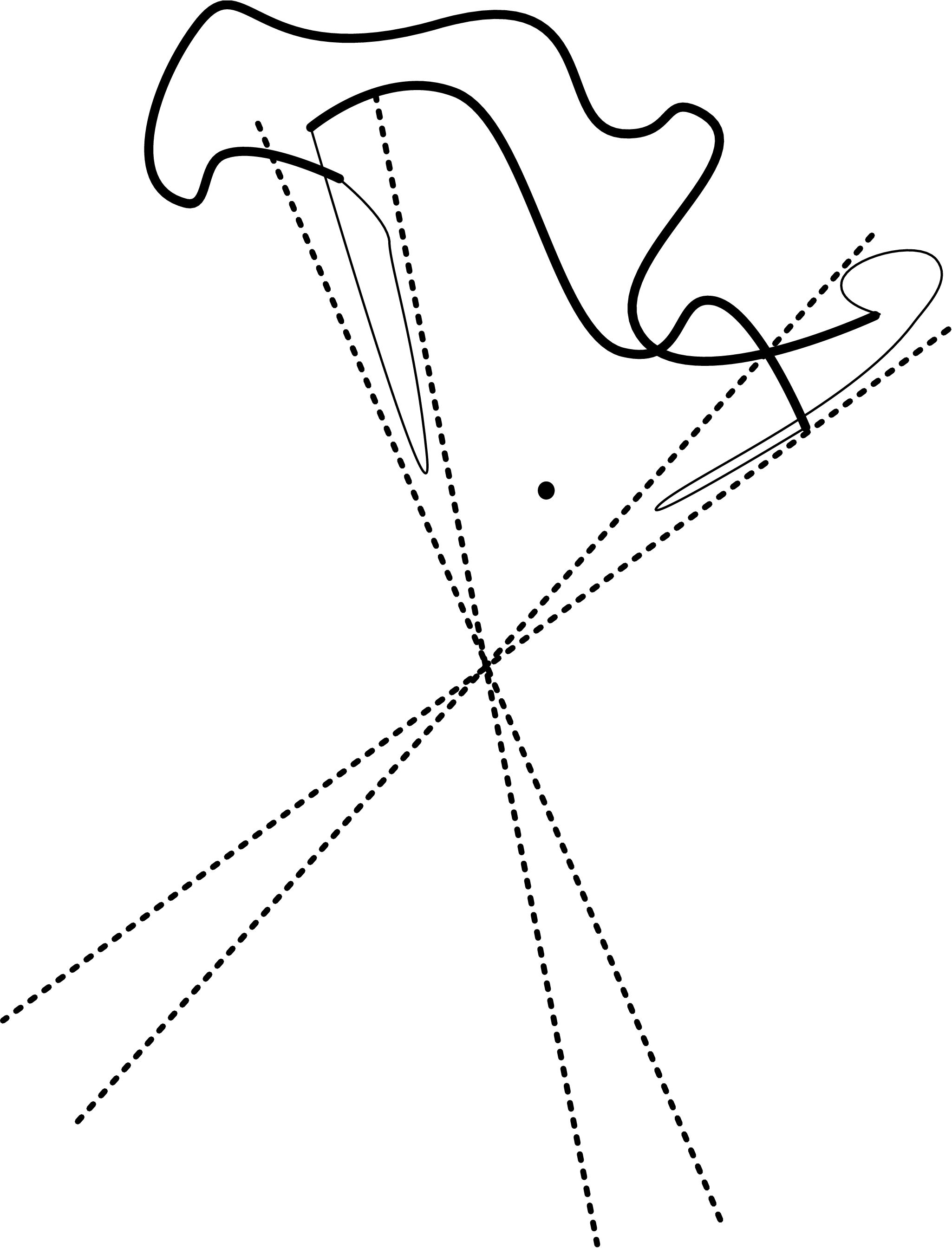}
	\caption{We have drawn in bold the two arcs $\pid \circ v_r( \{ \wt{R_r}, \wt{R_r}' \} \times [0,1])$ and with a finer nib $\pid \circ v_r( [\wt{R_r}, \wt{R_r}'] \times \{ 0 \} ) \subset W$ and $\pid \circ v_r( [\wt{R_r}, \wt{R_r}'] \times \{ 1 \} ) \subset e^{-i \theta}W$.  Their union gives the image of a loop in the plane.  The point $a \in \bC$ lies in the same connected component as $0$ of the complement of this loop.}
	\label{fig:bettercallsaul}
\end{figure}

We have illustrated this in Figure \ref{fig:bettercallsaul}.
	
	Since $\pid \circ v_r$ is holomorphic, it follows that the winding number of
	\[ \pi_d \circ v_r ( \{ \wt{R_r}, \wt{R_r}' \} \times [0,1] \cup [\wt{R_r}, \wt{R_r}'] \times \{ 0 , 1 \} ) \]
	around $a \in \bC$ is positive so, in particular, non-zero.  Then, since the origin $0 \in \C$ is in the same connected component of
	\[ \bC \setminus \pi_d \circ v_r ( \{ \wt{R_r}, \wt{R_r}' \} \times [0,1] \cup [\wt{R_r}, \wt{R_r}'] \times \{ 0 , 1 \} ) \]
	as $a \in \bC$ it follows that the image of $v_r$ intersects the diagonal $\Delta(\bC)$, which contradicts our hypotheses.
\end{proof}

	\begin{lem}
		[\bf Sphere bubbling]
		\label{lem:sphere}
		There is no sphere bubbling.
	\end{lem}
	
	\begin{proof}
	The argument is well-known.
	For sake of contradiction, a sphere bubble would entail the existence of a non-constant map of the Riemann sphere $v \colon \C \mathbb{P}^1 \to \bC^2$ which is $J$-holomorphic for some $J \in \cJ(\bC^2,\omega)$.
	Since $v$ is non-constant, its energy is positive.
	On the other hand, since $v$ is $J$-holomorphic, its energy equals $[\omega](v_*[\bC P^1])$.
	Since $H_2(\bC^2) = 0$, this value vanishes, and we reach a contradiction.
	\end{proof}

	\begin{lem}
		[\bf Disc bubbling]
		\label{lem:disc_bubbling}
		There is no disc bubbling.
	\end{lem}
	\begin{proof}
	The main case of disc bubbling to exclude will be when the bubbling happens at a point on $\Delta(\gamma)$.  Let us begin by arguing why the other cases cannot occur.  Note that it is not possible for a disc bubble to be disjoint from $\Delta(\bC)$ in both its interior and its boundary circle, since any loop on $(\gamma \times \gamma)\setminus \Delta(\gamma)$ either bounds a disc of zero symplectic area (and hence cannot bound a disc bubble), or has non-zero winding number around $\Delta(\bC)$ (and hence any disc bubble it bounds must intersect $\Delta(\bC)$ in its interior).  On the other hand, the local arguments of Lemma \ref{lem:touching_the_diagonal} can be applied \emph{mutatis mutandis} to show that a disc bubble cannot meet the diagonal at any point except that at which it bubbles off, and likewise neither can the resulting strip.
	
	Suppose then that for $r=1$, a disc bubbles off at $(p,p) \in \Delta(\gamma)$ and assume that this bubble meets the diagonal only at this point.  For a point $\star \in \partial \Sigma \subset \Sigma \subset \bC$, let $D_\star$ be the closed unit disc externally tangent to $\Sigma$ at $\star$.  We model the bubbling as a map
	\[ u_1 \colon \Sigma \cup D_\star \longrightarrow \bC^2 \]
	where $u_1 \vert_\Sigma$ satisfies the perturbed Cauchy-Riemann-Floer equation and $u_1 \vert_{D_\star}$ is $\Jstd$-holomorphic.
	
	Let $B_\epsilon$ be the closed disc of radius $\epsilon > 0$ centred at $\star \in \bC$.  Then for small enough $\epsilon$, $u_1 \vert_{B_\epsilon \cap \Sigma}$ satisfies the Cauchy-Riemann equation (since the Hamiltonian and its perturbation vanish near the boundary of the strip, and the almost-complex structure is $\Jstd$ near $\Delta(\bC)$).  This then allows us, for small enough $\epsilon$, to Schwarz-reflect $u_1\vert_{B_\epsilon \cap \Sigma}$ and $u_1\vert_{B_\epsilon \cap D_\star}$ across $\gamma \times \gamma$ to obtain two holomorphic discs in $\bC^2$ which each meet $\Delta(\bC)$ exactly in the point $(p,p)$.  By positivity of intersection their boundary circles must each have positive winding number around $\Delta(\bC)$.

	Now for $r<1$ we take a smoothly varying family of properly embedded arcs $A^1_r, A^2_r \subset \Sigma$ cobounding rectangles $\rho_r \subset \Sigma$ so that $u_r(A^1_r) \rightarrow u_1(\partial B_\epsilon \cap D_\star)$ and $u_r(A^2_r) \rightarrow u_1(\partial B_\epsilon \cap \Sigma)$.  If we have chosen $\epsilon$ small enough, then for $r$ sufficiently close to $1$, the rectangles $\rho_r$ may be Schwarz-reflected to give holomorphic cylinders in $\bC^2$, which degenerate to at $r=1$ to a pair of holomorphic discs meeting at the point $(p,p)$.
	
	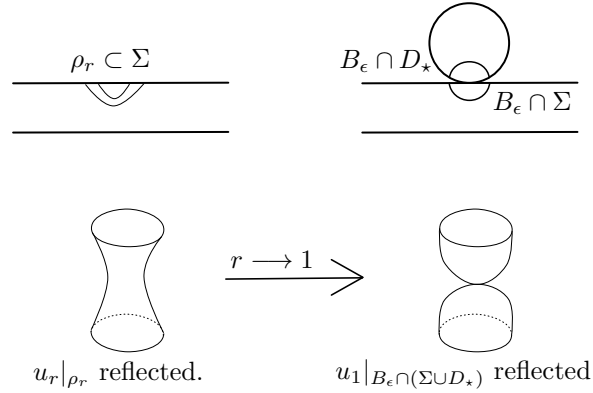
\begin{figure}
		\labellist
		\pinlabel {$B_\epsilon \cap D_\star$} at 1410 1100
		\pinlabel {$B_\epsilon \cap \Sigma$} at 1960 940
		\pinlabel {$\rho_r \subset \Sigma$} at 370 1100
		\pinlabel {$u_r\vert_{\rho_r}$ reflected.} at 400 -80
		\pinlabel {$u_1 \vert_{B_\epsilon \cap (\Sigma \cup D_{\star})}$ reflected} at 1700 -80
		\pinlabel {$r \longrightarrow 1$} at 980 350
		\endlabellist
		\centering
		\includegraphics[scale=0.1]{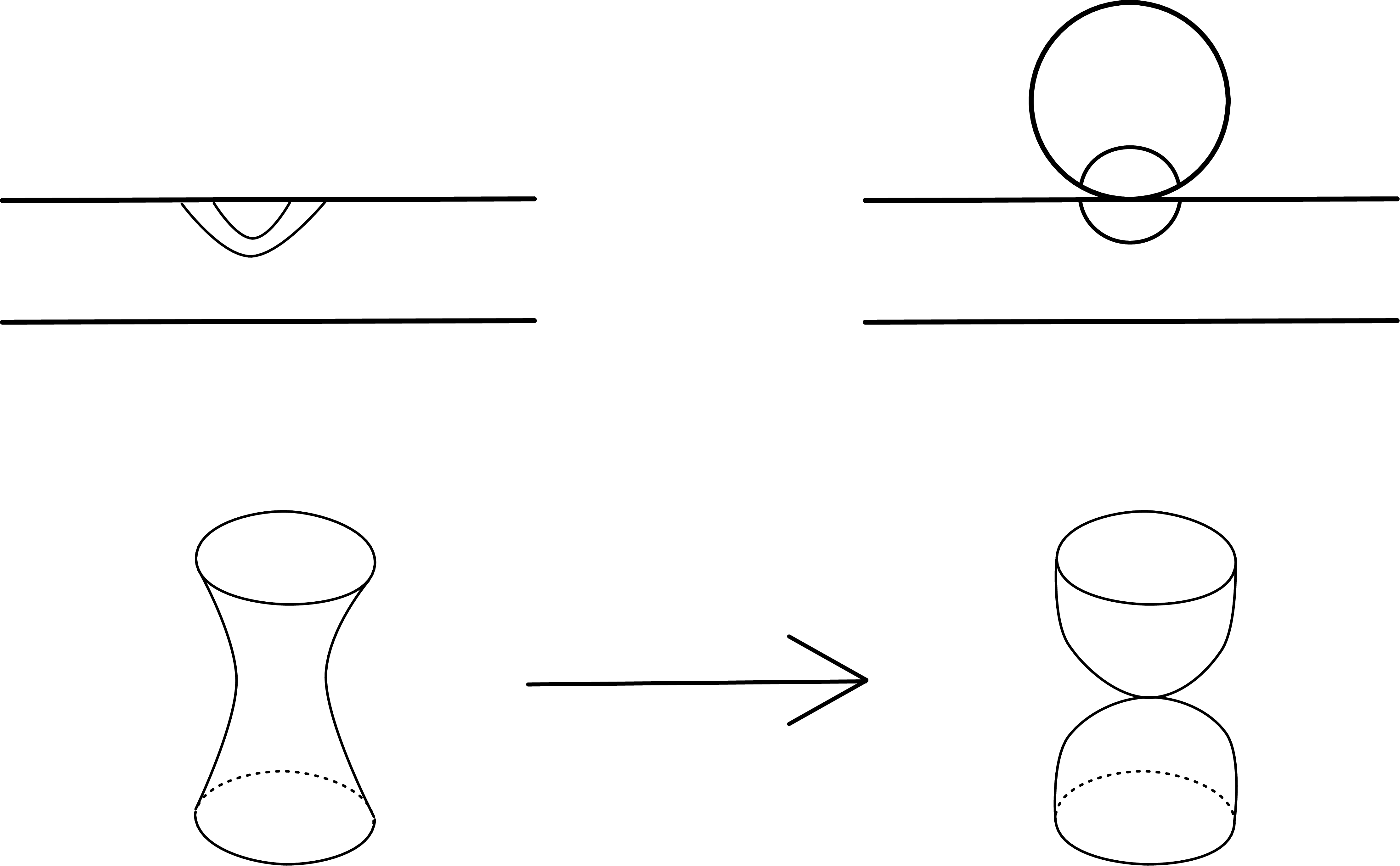}
		\vspace*{5mm}
		\caption{On the left we have shown the embedded rectangle $\rho_r \subset \Sigma$ and its image, (along with its Schwarz reflection) in $\bC^2$ -- this is a holomorphic cylinder.  On the right we have shown the bubbling degeneration as $r \rightarrow 1$.  The cylinder degenerates to the union of two holomorphic discs which meet at the point $(p,p)$ at which the bubble occurs.  These discs each meet the diagonal $\Delta(\bC)$ positively, so for $r$ close to $1$, the cylinders must also.}
		\label{fig:bubble}
	\end{figure}
	
	We illustrate the situation in Figure \ref{fig:bubble}.
	
	Since the winding numbers around $\Delta(\bC)$ of each boundary component of the cylinders is positive, it follows that the cylinders must intersect $\Delta(\bC)$ in their interiors.  Therefore $u_r$ must meet $\Delta(\bC)$ and we have a contradiction.
\end{proof}

\subsection{The continuation package.}
\label{subsec:continuation}
Once one has defined a Lagrangian Floer chain complex for a choice of Lagrangian, Hamiltonian, Hamiltonian perturbation, and almost-complex structure, there is a general procedure available for showing that its chain homotopy type is independent of choice of perturbation and almost-complex structure, and moreover for setting up isomorphisms between the homologies corresponding to different choice of Hamiltonian.  This general procedure involves counts of moduli spaces of strips in which the Hamiltonians and the almost-complex structures are now allowed to be strip-dependent.  The chain maps are sometimes known as \emph{continuation} maps, and we refer to the collection of strip counts from which one can define chain maps, chain homotopies, and so on as the `continuation package'.  For a readable account we refer the reader to \cite[Section 1.5]{aurouxintro}.

We begin by addressing chain maps in our setting.  Suppose that we are considering the admissible quadruples $(\gamma, \theta_1, h^1_t, J^1_t)$ and $(\gamma, \theta_2, h^2_t, J^2_t)$.  We choose a smooth monotonic function $\beta' \colon \bR \rightarrow [0,1]$ which is $0$ in a neighborhood of $(-\infty,0]$ and $1$ in a neighborhood of $[1, \infty)$.  Then we define the strip-dependent Hamiltonian $H^{12}_{st}$ by
\[ H^{12}_{st} = (1 - \beta'(s))\theta_1 H_t + \beta'(s)\theta_2 H_t {\rm .} \]

We also choose a smooth path of $J^{12}_{st}$ of almost-complex structures such that $J^{12}_{st} = J^1_t$ for $s \leq 0$ and $J^{12}_{st} = J^2_t$ for $s \geq 1$, and choose these $J^{12}_{st}$ to be standard $\Jstd$ again, as usual, sufficiently near the boundaries of the strip and the diagonal $\Delta(\bC)$.

The choices so far have been made without a view towards establishing transversality of moduli spaces.  When we were trying to achieve this in the case of the differential, we were careful with our choice of almost-complex structure $J_t$.  In the case of these continuation strips, the usual approach is to introduce a strip-dependent Hamiltonian perturbation.  For our purposes we wish to add a perturbation term of the form
\[ h^{12}_{st} \colon \bC^2 \longrightarrow \bR \]
in which we make the requirement that $h^{12}_{st}$ is zero outside the region $(s,t) \in (0,1) \times (0.1,0.9) \subset \Sigma$, $h^{12}_{st} = h^1_t$ when $s \leq 0$, $h^{12}_{st} = h^2_t$ when $s \geq 1$, and moreover we require that all the partial derivatives of $h^{12}_{st}$ vanish to infinite order at the diagonal $\Delta(\bC)$.  That such perturbations suffice for transversality is shown by Audin-Damian \cite[Section 11.1.b]{audindamian}.

We then define a chain map
\[ \JFC(\gamma, \theta_1, h^1_t, J^1_t) \longrightarrow \JFC(\gamma, \theta_2, h^2_t, J^2_t) \]
as a matrix with respect to the bases $\cG(\gamma,\theta_1, h_t^1)$ and $\cG(\gamma,\theta_2, h_t^2)$ in which the coefficients of the matrix are given by counts of $0$-dimensional moduli spaces of strips with Lagrangian boundary conditions on $\gamma \times \gamma$, satisfying the Cauchy-Riemann-Floer equation with respect to the Hamiltonian $H^{12}_{st} + h^{12}_{st}$ and the almost-complex structure $J^{12}_{st}$, limiting at either end to non-constant trajectories, and whose closures in $\bC^2$ are disjoint from $\Delta(\bC)$.

The proof that a map defined in such a way does define a chain map results from analysing the Gromov boundary degenerations of the $1$-dimensional moduli spaces of such strips.  Again, the main concern for us will be in verifying that such degenerations consist of concatenations of diagonal-avoiding strips.

Supposing for moment that one has successfully performed this verification, there remains the question of whether the chain maps thus defined are in fact chain-homotopy equivalences.  That brings us to the next component of the continuation package.  The continuation approach to proving this result is to make a count of strips (or now, rather, a count from among a $1$-parameter family of strips which come with strip-dependent Hamiltonians, almost complex structures, and Hamiltonian perturbations) of Maslov index $-1$ to define a chain homotopy $h$.  The boundary degenerations of the $1$-parameter family of Maslov index $0$ strips should then add up to $0$ modulo $2$, allowing one to verify that
${\rm id} + F \circ G + \partial \circ h + h \circ \partial = 0$ where ${\rm id}$ is the identity chain map, $F$ and $G$ are two chain maps constructed as above such that $F \circ G$ and $G \circ F$ are both defined, and $\partial$ is the Floer differential.

Again, the main question for us is whether this can all be carried out for diagonal-avoiding strips.  That is, we must make the same checks as before: verifying that a $1$-parameter family of such strips cannot wander up to and touch the diagonal, cannot bubble off a disc at the diagonal, and cannot break at the diagonal.  The proofs proceed as they did in verifying $\del^2=0$ in \Cref{subsec:differential}, but the case of diagonal breaking needs a little modification.

First, let us list some important features of the Hamiltonians and almost-complex structures that show up in our situation in the setting of the strips of the continuation package.  We begin with a pair of admissible quadruples $(\gamma, \theta_i, h_t^i, J^i_t)$ for $i=1,2$.
We write $h^{12}_{st}$ for the Hamiltonian perturbation, $H^{12}_{st}$ for the `prescribed' Hamiltonian (so that the Hamiltonian used in the Cauchy-Riemann-Floer equation is $H^{12}_{st} + h^{12}_{st}$), and $J_{st}$ for the almost complex structure.  We have the following:

\begin{itemize}
	\item For some $S \gg 0$ and for all $s < -S$ we have $H^{12}_{st} = \theta_1 H_t$, $h^{12}_{st} = h^1_t$,  and $J^{12}_{st} = J^1_t$, and for all $s > S$ we have $H^{12}_{st} = \theta_2 H_t$, $h^{12}_{st}= h^2_t$,  and $J^{12}_{st} = J^2_t$.
	\item For all fixed $s_0$, we have $H^{12}_{s_0 t} = \theta_0 H_t$ for some $\theta_0$ lying between $\theta_1$ and $\theta_2$.
	\item We have that $h^{12}_{st}$ and all its partial derivatives up to infinite order vanish at the diagonal $\Delta(\bC)$.
	\item We have $J^{12}_{st} = \Jstd$ and $h^{12}_{st} = 0$ for $0 \leq t < 0.1$ and $0.9 < t \leq 1$.
	\item We have $J^{12}_{st}(z,w) = \Jstd$ whenever $\vert z - w \vert < \w(\gamma)/2$.
\end{itemize}

\begin{lem}
[\bf Touching the diagonal, sphere bubbling, disc bubbling]
	\label{lem:cont_maps_no_touching}
	If $u_r$, $0 \leq r < 1$ is $1$-parameter family of continuation strips such that $u_r$ is diagonal-avoiding, then the Gromov limit $u_1$ cannot be a diagonal-intersecting strip, or a strip with a disc or sphere bubble.
\end{lem}
\begin{proof}
	The proofs of Lemmas \ref{lem:touching_the_diagonal}, \ref{lem:sphere}, and \ref{lem:disc_bubbling} carry over to this setting.
\end{proof}

For the case of diagonal breaking, we begin with a preparatory lemma which gives the reason that we have imposed the vanishing to infinite order of the partial derivatives of the perturbations $h_{st}$ at the diagonal.

\begin{lem}
	[Angular bound]
	\label{lem:boundedness_of_twist}
	Let $h_t \colon \bC^2 \rightarrow \bR$ be a time-dependent Hamiltonian all of whose partial derivatives vanish to infinite order at the diagonal $\Delta(\bC)$.
	
	We write $\pid \colon \bC^2 \rightarrow \bC \colon (z,w) \mapsto z-w$ for the projection.  Suppose that $C \subset \Delta(\bC)$ is a compact subset of the diagonal and for $\epsilon > 0$ write
	\[ C_\epsilon := \{ (z,w) \in \bC^2 : d((z,w), C) \leq \epsilon \} \]
	for the compact subset of points of distance at most $\epsilon$ from $C$.
	
	Let $0 < \theta < \pi$ and $\delta>0$.  Then by choosing $\epsilon$ small enough we can ensure that whenever $(z,w) \in C_\epsilon \setminus \Delta(\bC)$ we have
	\[ \arg (\pid(z,w)) - \arg(\pid(\Phi^1_{\theta H_t + h_t}(z,w))) \in (\theta - \delta, \theta + \delta) {\rm .} \]
\end{lem}

\begin{proof}
	First note that the result certainly follows if $h_t$ is identically zero, since in this case the Hamiltonian flow is just rotation about the diagonal by angle $\theta$.  The lemma is stating that, close enough to the diagonal, the perturbed Hamiltonian looks as near enough to such a rotation as we like.
	
	By compactness and second-order vanishing of the derivatives, we can find some $M > 0$ such that
	\[ \vert d h_t \vert < M \vert (z,w) \vert^2 \]
	whenever $(z,w) \in C_\epsilon$.  Hence we see that the Hamiltonian vector field $X_{h_t}$ associated to $h_t$ satisfies
	\[ \vert X_{h_t} \vert < M \vert (z,w) \vert^2 {\rm .} \]
	
	Possibly shrinking $C_\epsilon$ if necessary by intersection with a closed tubular neighborhood of $\Delta(\bC)$, we thus ensure that the $\partial/\partial \theta$ component of $X_{h_t}$ is of norm at most $\delta \vert z - w \vert$ in magnitude.  Shrinking again if necessary to a neighborhood not leaving the previous neighborhood under the flow, we are done since
	\[ X_{\theta H_t + h_t} = X_{\theta H_t} + X_{h_t} {\rm .} \]
\end{proof}

\begin{lem}
[\bf Breaking at the diagonal]
	\label{lem:cont_diagonal_breaking}
	Suppose for $0 \leq r <1$ that $u_r \colon \Sigma \rightarrow \bC^2$ is a $1$-parameter family of diagonal-avoiding continuation strips.  Then the Gromov limit at $r=1$ is a never a broken strip that somewhere meets the diagonal.
\end{lem}

\begin{proof}
	As in Lemma \ref{lem:diagonal_breaking}, we may rule out meeting the diagonal anywhere, except for the possibility of one of the breaking points being at the diagonal.  The rest of the proof will take the proof of Lemma \ref{lem:diagonal_breaking} as its guide, but some modification will be required.
	
	Let us assume for concision of notation that we have a single breakpoint, as this assumption will not materially affect the argument.
	
	Then we have, say $u_1 = u_1^- \# u_1^+$ so that $u_1^-(s,t), u_1(-s,t) \rightarrow (p,p) \in \Delta(\gamma)$ as $s \rightarrow \infty$.
	
	The strip $u_1^-$ satisfies a Cauchy-Riemann-Floer equation in which the total Hamiltonian $\cH^-_{st}$ is only time-dependent for large enough values of $s \gg 0$.  Within distance $\w(\gamma)/2$ of the diagonal $\Delta(\bC)$, this time-dependent Hamiltonian agrees with $\theta^- H_t$ for some angle $0 < \theta^- < \pi$.  Similarly, the strip $u_1^+$ satisfies a Cauchy-Riemann-Floer equation in which the Hamiltonian agrees with $\theta^+ H_t$ for small enough $s \ll 0$  and within $\w(\gamma)/2$ of $\Delta(\bC)$.
	
	We proceed as in Lemma~\ref{lem:diagonal_breaking} and choose 
	a large $R \gg 0$, such that the reformulated versions $v_1^-$ and $v_1^+$ of $u_1^-$ and $u_1^+$ respectively (following Subsection \ref{subsec:reformulation}) are both holomorphic with respect to $\Jstd$ when restricted to $[R,\infty) \times [0,1]$ and $(-\infty, -R] \times [0,1]$ respectively.  Furthermore $R$ is chosen large enough so that $v_1^-$ and $v_1^+$ satisfy the following:
\begin{align*}
\pid \circ v_1^-( [R,\infty) \times \{ 0 \} ) \cup  \pid \circ v_1^+( (-\infty, -R] \times \{ 0 \} ) &\subset W {\rm ,} \\
\pid \circ v_1^-( [R,\infty) \times \{ 1 \} )  &\subset e^{-i \theta^-}W {\rm ,} \\
{\rm and} \,\, \pid \circ v_1^+( (-\infty, -R] \times \{ 1 \} ) &\subset e^{-i \theta^+}W {\rm .}
\end{align*}
And we have $W \cap e^{-i \theta^-}W = W \cap e^{-i \theta^+}W = e^{-i \theta^-}W \cap e^{-i \theta^+}W = \{ 0 \}$.

	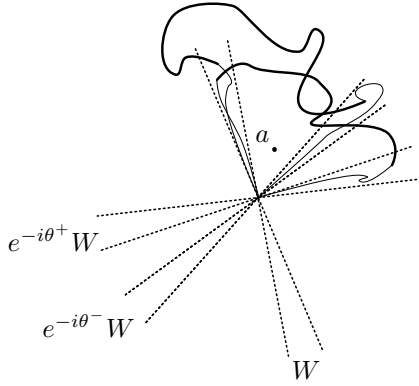
\begin{figure}
	\labellist
	\pinlabel {$a$} at 625 830
	\pinlabel {$W$} at 790 -50
	\pinlabel {$e^{-i \theta^-} W$} at -30 125
	\pinlabel {$e^{-i \theta^+} W$} at -160 450
	\endlabellist
	\centering
	\includegraphics[scale=0.1]{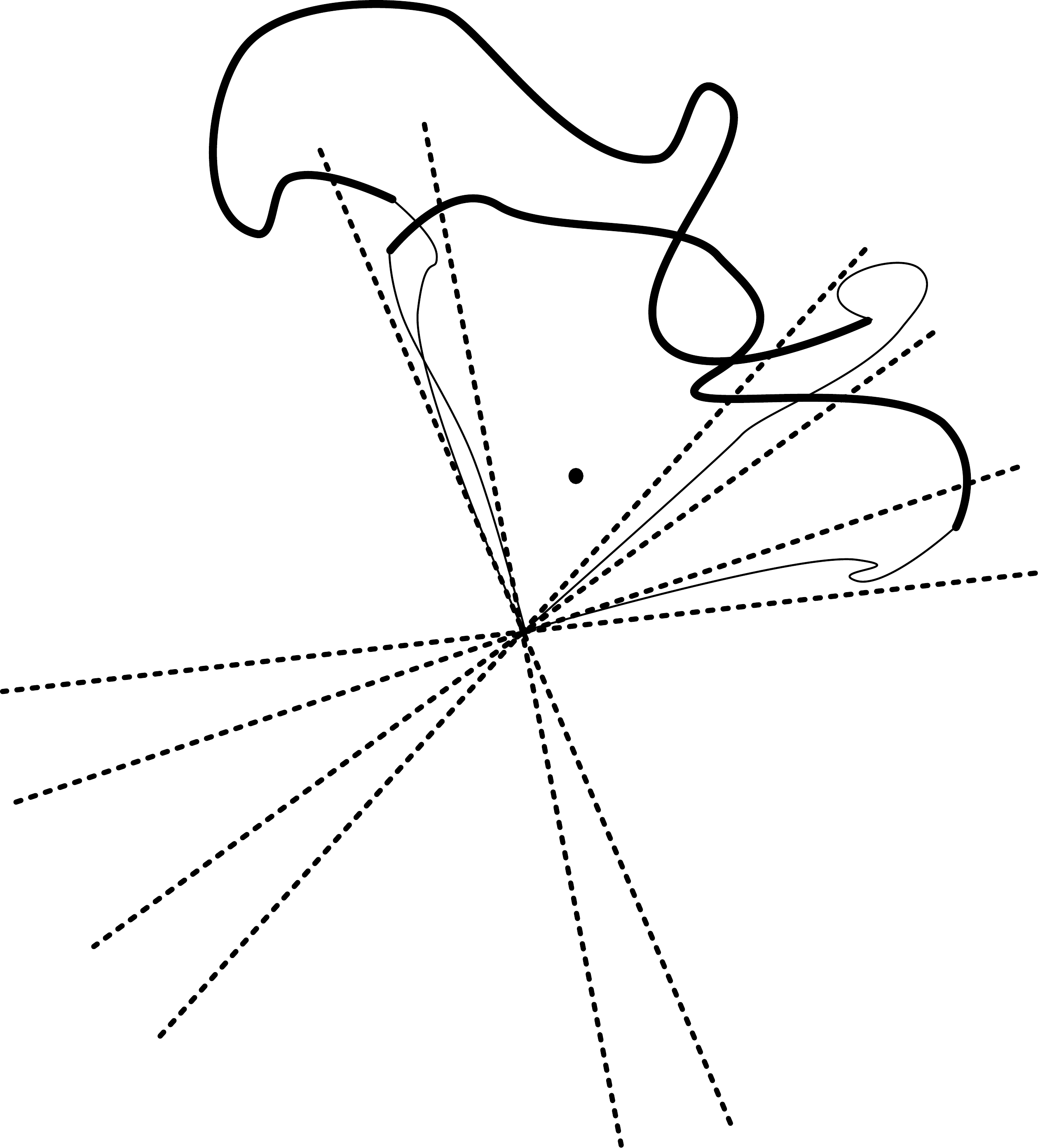}
	\caption{We show the three regions $W$, $e^{-i \theta^-}W$, and $e^{-i \theta^+}W$ of the plane.  Within them we have drawn in bold both $\pid \circ v_1^-(\{R\} \times [0,1])$ and $\pid \circ v_1^+(\{-R\} \times [0,1])$.  Using a finer nib we have drawn $\pid \circ v_1^-([R, \infty) \times \{ 0,1 \})$ and $\pid \circ v_1^+((-\infty,-R] \times \{ 0,1 \})$.  By adding in $\{ 0 \} = \pid \circ v_1^-(\{\infty\} \times [0,1])$ to $\pid \circ v_1^-(\{R\} \times [0,1])$ and $\pid \circ v_1^-([R, \infty) \times \{ 0,1 \})$, we obtain the image of a closed loop $\Lambda$ (possibly self-intersecting) in the plane.  The loop $\Lambda$ winds around the point in the plane labelled with an $a$.}
	\label{fig:bojackhorseman}
\end{figure}

We illustrate the situation in Figure \ref{fig:bojackhorseman}.
An important point for us to note, looking at Figure \ref{fig:bojackhorseman}, is that $\pid \circ v_1^-( [R,\infty) \times \{ 0 \} )$ and  $\pid \circ v_1^+( (-\infty, -R] \times \{ 0 \} )$ both lie in the same component of $W \setminus \{ 0 \}$.  Similarly, if $\pid \circ v_1^+( (-\infty, -R] \times \{ 1 \} )$ lies in component $C$ of $e^{-i \theta^+} W \setminus \{ 0 \}$, then $\pid \circ v_1^-( [R,\infty) \times \{ 1 \} )$ lies in component $e^{i(\theta^+ - \theta^-)}C$ of $e^{-i \theta^-} W \setminus \{ 0 \}$.  These points both follow from the hypothesis that $u_r$ avoids the diagonal $\Delta(\bC)$ for $r$ close to $1$.

As in the proof of Lemma \ref{lem:diagonal_breaking}, we have also chosen a point $a \in \bC$, in this case lying outside $W \cup e^{-i \theta^-} W \cup e^{-i \theta^+} W$, which is at a distance to the origin of less than half the minimal distance from the origin of $\pid \circ v_1^-(\{R\} \times [0,1]) \cup \pid \circ v_1^+(\{-R\} \times [0,1])$.  The point $a \in \bC$ is chosen so that the loop
\[ \Lambda^- := \pid \circ v_1^-(\{R\} \times [0,1]) \cup \pid \circ v_1^-( [R,\infty) \times \{ 0 \} ) \cup \pid \circ v_1^-([R, \infty) \times \{ 0,1 \}) \cup \{ 0 \} \]
has non-zero winding number around $a$.  This means that we must have $a \in \pid \circ v_1^- ((R, \infty) \times (0,1))$, so that in fact the winding number is strictly positive due to the holomorphicity of $v_1^-$.  Since $v_1^+$ is holomorphic when restricted to $(-\infty, -R] \times [0,1]$, it follows that the loop
\[ \Lambda^+ := \pid \circ v_1^+(\{-R\} \times [0,1]) \cup \pid \circ v_1^+( (-\infty,-R] \times \{ 0 \} ) \cup \pid \circ v_1^+((-\infty, -R] \times \{ 0,1 \}) \cup \{ 0 \} \]
(when given the natural orientation) must have non-negative winding number around $a$.

We now turn to Figure \ref{fig:malcolminthemiddle}.

	Again, similarly to our approach in the proof of Lemma \ref{lem:diagonal_breaking} we now wish to consider approximations to the broken strips by reformulations $v_r$ of strips $u_r$ following Subsection \ref{subsec:reformulation}.  A subtlety here is that the reformulation will \emph{not} necessarily result in holomorphic strips (with respect to a reformulated almost complex structure) because the Hamiltonian on $u_r$ is possibly strip-dependent and not merely time-dependent.  Our interest is less analytic than topological, however.
	
	We observe that we may choose $\wt{R_r}, \wt{R_r}' \in \bR$ with $\wt{R_r} < \wt{R_r}'$ so that
	\[ v_r(\partial( [\wt{R_r},\wt{R_r}'] \times [0,1])) \rightarrow v_1^-( \partial( [R,\infty]\times[0,1] )) \cup v_1^+(\partial( [-\infty,R] \times [0,1] )) \]
	and, in particular,
	\[ \pid \circ v_r(\partial([\wt{R_r},\wt{R_r}' ] \times [0,1]) ) \rightarrow \Lambda^- \cup \Lambda^+ \,\, {\rm as} \,\, r \rightarrow 1 {\rm .} \]
	We make our choices so that for all $r$ sufficiently close to $1$ we have the following
\begin{itemize}
	\item $\pid \circ v_r( \{ \wt{R_r}, \wt{R_r}' \} \times [0,1])$ is distance at least $3\delta/4$ from the origin,
	\item $\pid \circ v_r( [\wt{R_r}, \wt{R_r}'] \times \{ 0 \} ) \subset ({\rm a} \,\, {\rm component} \,\, {\rm of} \,\, W \setminus \{0\})$,
	\item and $\pid \circ v_r( [\wt{R_r}, \wt{R_r}'] \times \{ 1 \} ) \subset ({\rm a} \,\, {\rm component} \,\, {\rm of} \,\, \bigcup_{0 \leq \theta \leq \theta^+ - \theta^-} e^{i (\theta - \theta^+)}W \setminus \{0\})$.
\end{itemize}

The first and second points follow simply from continuity, while the third follows from Lemma \ref{lem:boundedness_of_twist}\footnote{Technically from an $r$-dependent version of Lemma \ref{lem:boundedness_of_twist}, since the perturbation $h^r_{st}$ may have smooth dependence on $r$.  But this is again a simple application of compactness, so we preferred to relegate mention of it to this footnote.}, so long as one is working within a neighborhood $C_\epsilon$ as in the statement of that Lemma.  This can be achieved by starting the proof of the current lemma again, but with a possibly larger value of $R$ to begin with so that
\[ v_1^-(\partial([R,\infty) \times [0,1])) \,\,\, {\rm and} \,\,\, v_1^+(\partial((-\infty, -R] \times [0,1]) ))\]
both lie in such a neighborhood.

\begin{figure}
	\labellist
	\pinlabel {$a$} at 625 830
	\pinlabel {$W$} at 790 -50
	\pinlabel {$e^{-i \theta^-} W$} at -30 125
	\pinlabel {$e^{-i \theta^+} W$} at -160 450
	\endlabellist
	\centering
	\includegraphics[scale=0.1]{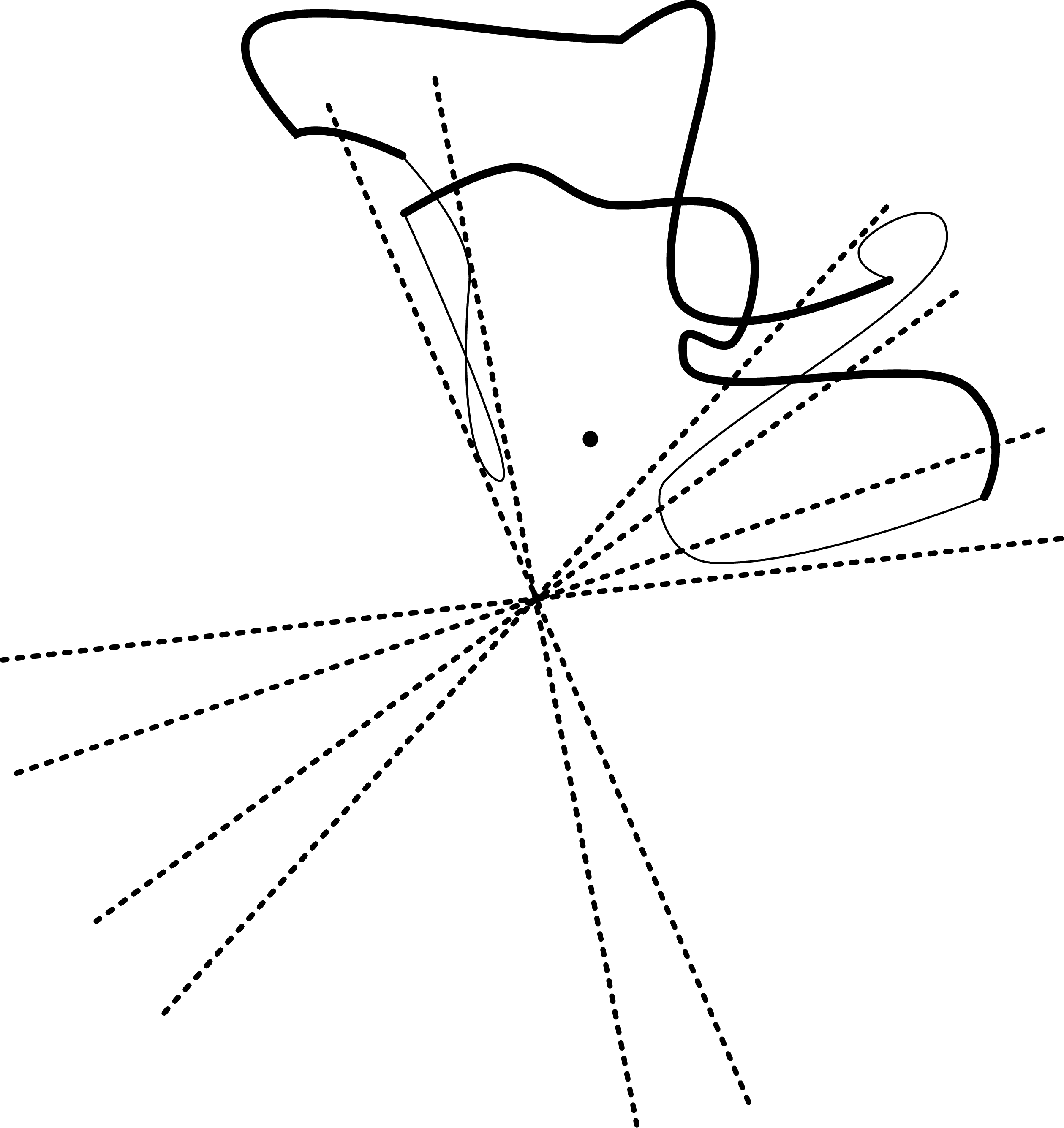}
	\caption{We have drawn in bold the two arcs $\pid \circ v_r( \{ \wt{R_r}, \wt{R_r}' \} \times [0,1])$ and with a finer nib $\pid \circ v_r( [\wt{R_r}, \wt{R_r}'] \times \{ 0 \} ) \subset W$ and $\pid \circ v_r( [\wt{R_r}, \wt{R_r}'] \times \{ 1 \} )$.  Their union gives the image of a loop in the plane.  The point $a \in \bC$ lies in the same connected component as $0$ of the complement of this loop.}
	\label{fig:malcolminthemiddle}
\end{figure}
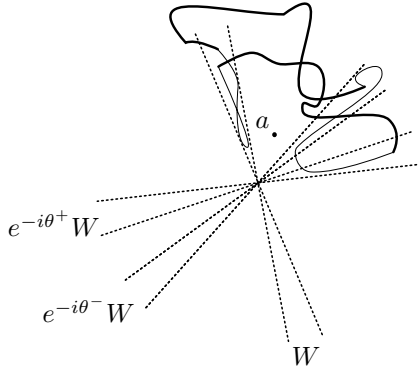

Since $0$ and $a$ lie in the same component of $\bC$ as the complement of the loop $\pid \circ v_r(\partial([\wt{R_r},\wt{R_r}' ] \times [0,1]) )$, it follows that the winding number of this loop is the same about $0$ as about $a$.  On the other hand, for $r$ sufficiently close to $1$, this winding number agrees with the winding number of $\Lambda^- \cup \Lambda^+$ around $a$, which is positive.  Thus we must have that
\[ v_r (\partial((\wt{R_r},\wt{R_r}' ) \times (0,1)) ) \cap \Delta(\bC) \not= \emptyset {\rm ,} \]
giving a contradiction.
\end{proof}

\subsubsection{Conclusion.}
We have shown that the continuation package of maps continues to function without modification when restricted to our setting of diagonal-avoiding strips.  Thus we have well-defined homology groups $\JF(\gamma, \theta)$ for pairs of a real analytic Jordan curve and angle $0 < \theta < \pi$.  These are defined as limits of directed systems of homology groups depending on choices of analytic data.  We further have well-defined isomorphisms $\JF(\gamma, \theta_1) \rightarrow \JF(\gamma, \theta_2)$ for any pairs $0 < \theta_1, \theta_2 < \pi$.

\subsubsection{An observation and a warning.}
\label{sec:warning}
We close this subsection with an observation about the proof of Lemma \ref{lem:cont_diagonal_breaking}: that it was important for the winding number arguments that we had $0 < \theta^-, \theta^+ < \pi$.
There is nothing to stop the reader defining Jordan Floer homology groups for pairs $(\gamma, \theta)$ in which $\theta < 0$ or $\theta > \pi$, but they should be wary of an attempt, for example, to define a continuation map following the principles above to give an isomorphism $\JF(\gamma, - \theta) \rightarrow \JF(\gamma,  \theta)$ for $0 < \theta < \pi$.
Such an attempt is doomed to fail simply by virtue of the support of the homologies in different Maslov degrees: on the left the support is in degrees $0,1$, while on the right it is in degrees $1,2$, as we shall see in \Cref{subsec:morse}.
In this paper we only use angles within the range $0 < \theta < \pi$, so we do not investigate this further here.

\subsection{Comparison with Morse homology.}
\label{subsec:morse}
The existence of the chain homotopy equivalences of the previous section establishes that the chain homotopy type of $\JFC_*(\gamma,\theta,h_t,J_t)$ depends only on $\gamma$.
By a familiar argument originating with Floer \cite[Theorem 2]{floer_witten}, a judicious choice of data identifies this complex with a Morse chain complex from which we may compute the isomorphism type of the homology group.
In our setting, we obtain an identification with a Morse-Bott chain complex of the pair $(\gamma \times \gamma, \Delta(\gamma)) \approx (S^1 \times S^1, \Delta(S^1))$, which yields:

\begin{thm}
[Isomorphism type]
\label{thm:jf}
For every admissible quadruple $(\gamma,\theta,h_t,J_t)$,
we have
\[
\JF_*(\gamma,\theta) := H_*(\JFC_*(\gamma,\theta,h_t,J_t)) \approx (\bF_2)_{(2)} \oplus (\bF_2)_{(1)}
\]
\end{thm}

In preparation, we review an argument of Oh \cite{oh1996} which leads to a spectral sequence from the Morse homology group $HM_*(L_0)$ to the self-Floer homology group $\mathrm{HF}_*(L_0,L_0)$.
With $\gamma$ fixed, choose the angle $\theta > 0$ sufficiently small so that the Lagrangian $R_\theta(\gamma \times \gamma)$ is contained within a Weinstein neighborhood $N$ of $L_0 = \gamma \times \gamma$ in $\bC^2$.
If we perturb $R_\theta(\gamma \times \gamma)$ by a small Hamiltonian into $L_1'$ made transverse to $L_0$ and still contained within $N$, then there exists a small energy $\hbar > 0$ such that every strip $u$ counted by the differential in the usual Floer chain complex $CF_*(L_0,L_1',J_t)$ with energy $E(u) < \hbar$ has image contained in $N$.
Moreover, for a suitable choice of almost-complex structure $J_t$, metric $g$ on $L_0$, and perturbation $h$, the assignment $u \mapsto u(s,0)$ puts these strips into one-to-one correspondence with Morse trajectories counted by the differential on the Morse chain complex $CM_*(L_0,\overline{H}+h,g)$.
Here $\overline{H}$ denotes the restriction of our Hamiltonian $H(z,w) = \frac14 |z-w|^2$ to $L_0$ and $h : L_0 \to \bR$ is chosen to make $\overline{H} + h$ a Morse function.

The higher differentials in the Oh spectral sequence count strips of higher energy in the Floer chain complex.
By contrast, counting the low-energy strips in our setting recovers the entire Jordan Floer chain complex.

\begin{proof}
We sketch the modification to Oh's argument required in our setting.
We work with the Lagrangian $L_1 = \Phi_{\theta H_t + h_t}^{-1}(\gamma \times \gamma)$ for an admissible choice of $h_t$, chosen sufficiently small.
In this case, as before, all low-energy strips in $\cM^\Delta_1(\gamma,\theta,h_t,J_t)$ are contained in $N$.
Now, however, the assignment $u \mapsto u(s,0)$ put these strips into one-to-one correspondence with Morse trajectories between nondegenerate critical points counted by the differential on the Morse-Bott chain complex $CM_*(L_0,\overline{H}+h,g)$.
Here $h$ is chosen so that $\overline{H} + h$ is a non-negative Morse-Bott function, its nondegenerate critical points are the transverse points of $L_0 \cap L_1$, and $\Delta(\gamma) = \overline{H}^{-1}(0) = (\overline{H}+h)^{-1}(0)$ is a circle of critical points.
It follows that we recover precisely the trajectories counted by the differential on the relative Morse-Bott chain complex $CM_*(L_0,\Delta(L_0),\overline{H}+h,g)$.
As noted, its homology is nothing other than $H_*(S^1 \times S^1, \Delta(S^1);\bF_2) \approx (\bF_2)_{(2)} \oplus (\bF_2)_{(1)}$.

It remains to explain why all of the strips in $\cM^\Delta_1(\gamma,\theta,h_t,J_t)$ have low energy, so that in fact $\JFC_*(\gamma,\theta,h_t,J_t)$ is isomorphic to the relative Morse-Bott complex.
As we explain in \Cref{sec:actions_and_rectangles}, each generator $\tau \in \cG(\gamma,\theta,h_t)$ has an associated {\em action} $\cA(\tau) = \cA_{\theta H_t + h_t}(\tau)$.
Suppose that $u \in \cM^\Delta(\gamma,\theta,h_t,J_t)$ is a strip from $\tau$ to $\tau'$.
Then $u \# \widehat{\tau}$ is a capping of $\tau'$ disjoint from $\Delta(\bC)$, so $[u \# \widehat{\tau}] = [\widehat{\tau'}]$.
It follows that $E(u) = \int u^* \omega = \cA(\tau')-\cA(\tau)$.
The actions of the trajectories in $\cG(\gamma,\theta,h_t)$ can be made uniformly and arbitrarily close to 0 by choosing both $h_t$ and $\theta$ sufficiently close to 0.
It follows that all strips in $u \in \cM^\Delta(\gamma,\theta,h_t,J_t)$ have sufficiently low energy $E(u) < \hbar$ to guarantee that their images are contained in $N$.
\end{proof}

Lastly, we remark that if $H$ restricts to a Morse function on $\gamma \times \gamma$, then the generating set of $CM_*(\gamma \times \gamma, \Delta(\gamma),\overline{H},g)$ admits a nice interpretation.
Namely, it consists of the {\em binormals} of $\gamma$, i.e. pairs $(z,w) \in \gamma \times \gamma - \Delta(\gamma)$ such that the tangent lines to $\gamma$ at $z$ and $w$ are perpendicular to the line segment $\overline{zw}$.
Thus we obtain a limiting group $JF(\gamma,0,H)$ which we may regard as the {\em Jordan Morse} homology of $\gamma$.

\section{Actions and rectangles.}
\label{sec:actions_and_rectangles}
Suppose that $\tau \in \Omega(\gamma \times \gamma)$ (the space of paths starting and ending on $\gamma \times \gamma$) and that $\tau$ is disjoint from the diagonal $\Delta(\bC)$.
Its {\em action} with respect to a time-dependent Hamiltonian $\cH_t$ is given by
\[
\cA_{\cH_t}(\tau) = \int_0^1 \cH_t \circ \tau(t) dt - \int_{[0,1]^2} {\widehat{\tau}}^*\omega
\]
where $\widehat{\tau}$ denotes a preferred capping of $\tau$.

We shall mostly be concerned with the action of non-constant trajectories $\tau \in \cG(\gamma, \theta,0)$ of the Hamiltonian $\theta H_t$.  Recall that these correspond to (nondegenerate) inscribed $\theta$-rectangles $Q \subset \gamma$ (we think of a rectangle as its set of vertices).

We now look at an example in order to get a handle on what the action $\cA_{\theta H_t} (\tau_Q)$ is telling us about the inscribed rectangle $Q$.  Figure \ref{fig:icecreamaction} shows an example of an {\em elegantly} inscribed $\theta$-rectangle: the Jordan curve is isotopic through Jordan curves into the rectangle $Q$, keeping the vertices of $Q$ fixed.
The computation of the action more involved for an inelegantly inscribed rectangle, but the case of an elegant inscription is nevertheless instructive.

\begin{figure}
	\labellist
	\pinlabel {$\gamma$} at 350 150
	\pinlabel {$\theta$} at 205 255
	\pinlabel {$r_1$} at 345 270
	\pinlabel {$r_2$} at 220 385
	\pinlabel {$r_3$} at -10 165
	\pinlabel {$r_4$} at 125 50
	\endlabellist
	\centering
	\includegraphics[scale=0.3]{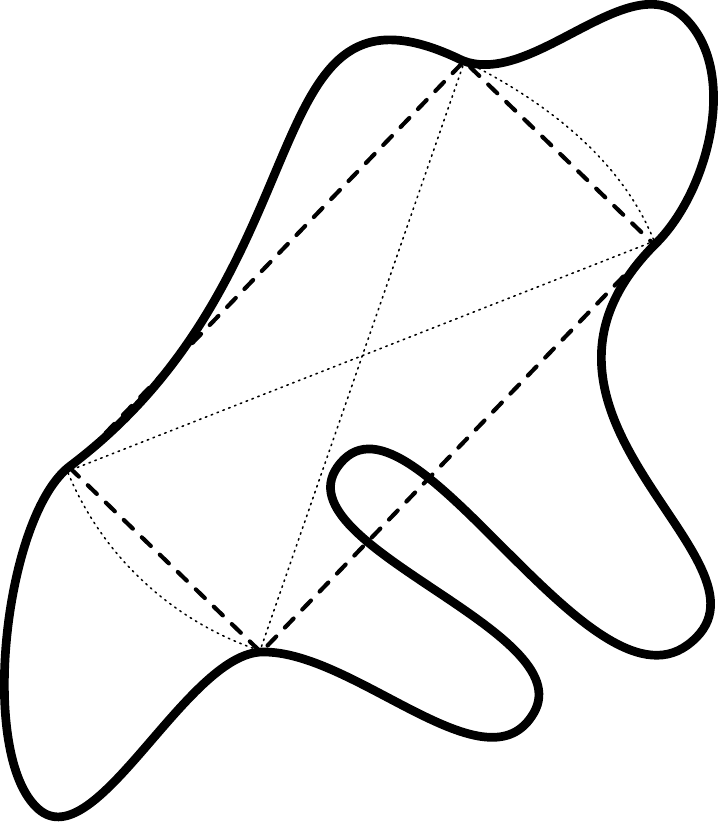}
	\caption{An example of an inscribed rectangle $Q = \{ r_1, r_2, r_3, r_4 \} \subset \gamma \subset \bC$ of aspect angle $\theta$.  We have included the diagonals of the rectangle and two circular arcs of diameter equal to the diameter of the rectangle.}
	\label{fig:icecreamaction}
\end{figure}
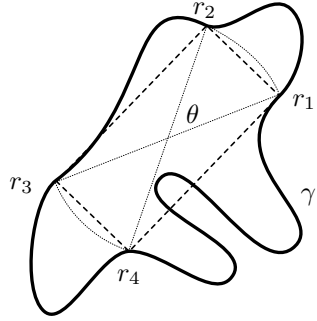

The action $\cA_{\theta H_t}(\tau_Q)$ is the sum of two terms.  Notice that $\theta H_t(\tau_Q(t)) = \theta  \beta(t) \rad(Q)^2$, where $\rad(Q)$ denotes half the length of the diagonal of $Q$.

The first term is just
\begin{align*}
\int_0^1 \theta H_t \circ \tau_Q(t) dt &= \theta \rad(Q)^2 \int^1_0 \beta(t) dt = \theta\rad(Q)^2 {\rm ,}
\end{align*}
and this is nothing more than the total area of the two regions shown in Figure \ref{fig:first_term_of_action}.

\begin{figure}
	\labellist
	\endlabellist
	\centering
	\includegraphics[scale=0.3]{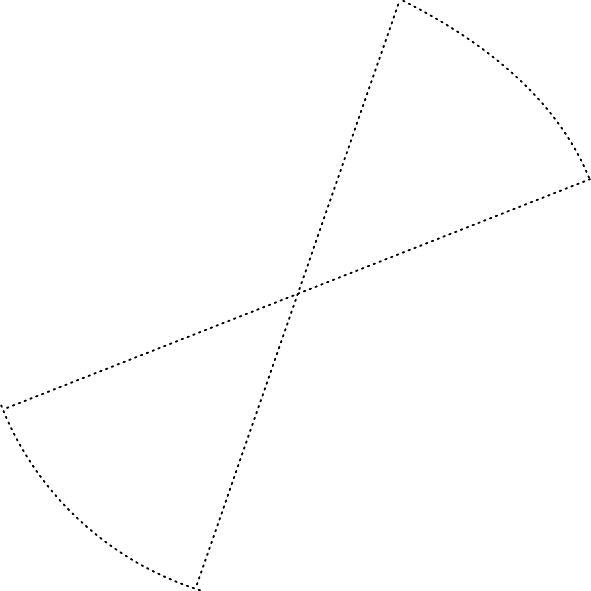}
	\caption{The first term in the formula for the action $\cA(\tau_Q)$ is the total area of the regions shown.}
	\label{fig:first_term_of_action}
\end{figure}

To compute the second term in the formula for the action, we first choose a preferred capping $\widehat{\tau_Q}$.
In \Cref{fig:first_term_of_action}, there is an embedded disk $D_1$ cobounded by the circular arc from $r_1$ to $r_2$ and the short arc of $\gamma$ from $r_1$ to $r_2$.
From it we obtain a map $\phi_1 : [0,1] \times [0,1] \to \bC$ with image $D_1$ such that $\phi_1(t,0)$ parametrizes the circular arc from $r_1$ to $r_2$ at constant speed for $t \in [0,1]$, $\phi_1(s,0) = r_1$, $\phi_1(s,1) = r_2$, and $\phi_1(1,t) \in \gamma$ for all $s,t \in [0,1]$.
There is another embedded disk $D_2$, giving rise to a map $\phi_2$, with the points $r_3,r_4$ in place of $r_1,r_2$.
Then $\widehat{\tau_Q} = (\phi_1,\phi_2)$ gives a capping of $\tau_Q$ with complementary boundary $P \subset \gamma \times \gamma$, and it is disjoint from $\Delta(\bC)$, because the images of $\phi_1$ and $\phi_2$ are the disjoint disks $D_1$ and $D_2$.
A similar method for constructing a preferred capping works with minor modification for any elegantly inscribed rectangle.



Continuing with the computation of the second term of the action, we consider the coordinate projections
\[ \pi_1 \colon \bC^2 \longrightarrow \bC \colon (z,w) \longmapsto z, \,\, {\rm and} \,\, \pi_2 \colon \bC^2 \longrightarrow \bC \colon (z,w) \longmapsto w {\rm .} \]
We have that
\[
\int_{[0,1]^2} \widehat{\tau}^* \omega = \int_{[0,1]^2} (\pi_1\circ \widehat{\tau_Q})^* dxdy + \int_{[0,1]^2}  (\pi_2\circ \widehat{\tau_Q})^* dxdy
\]
and this is nothing more than the sum of the areas bounded by the loops
\[ \pi_1(\tau_Q \cup P) \,\, {\rm and} \,\, \pi_2(\tau_Q \cup P) {\rm.} \]
We show these in Figure \ref{fig:second_term_of_action}.
\begin{figure}
	\labellist
	\endlabellist
	\centering
	\includegraphics[scale=0.3]{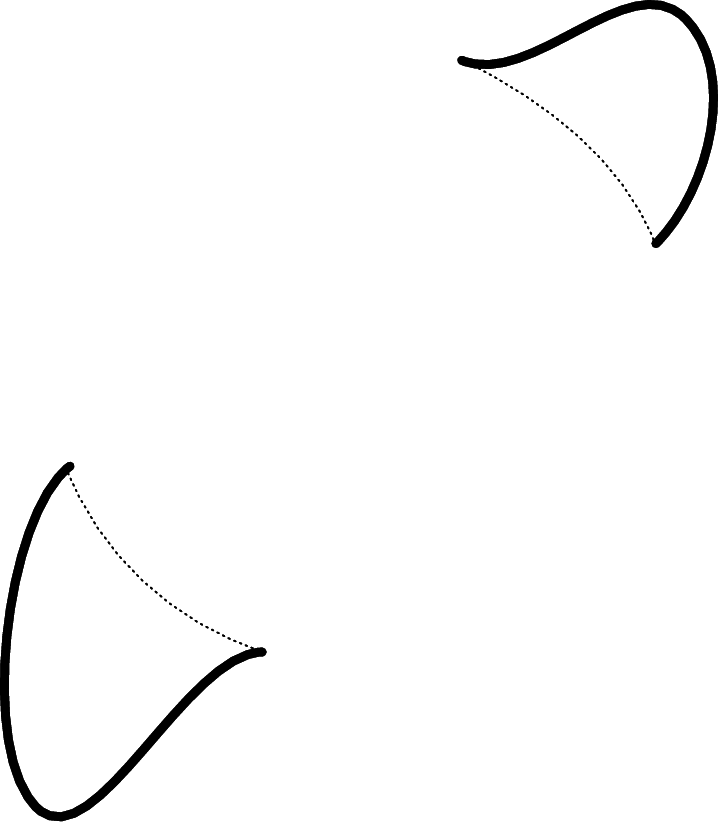}
	\caption{The second term in the formula for the action $\int_{[0,1]^2} \widehat{\tau}^* \omega$ is the total area of the two regions shown.}
	\label{fig:second_term_of_action}
\end{figure}

The second term of the action functional comes with a negative sign, so one needs to think carefully about orientations in order to make sense of it.  Once that thinking has been undertaken and successfully completed, the conclusion is that the action $\cA_{\theta H_t}(\tau_Q)$ is the (positive) area of the `double ice-cream cone' regions shown in Figure \ref{fig:doubleicecream}.

\begin{figure}
	\labellist
	\endlabellist
	\centering
	\includegraphics[scale=0.3]{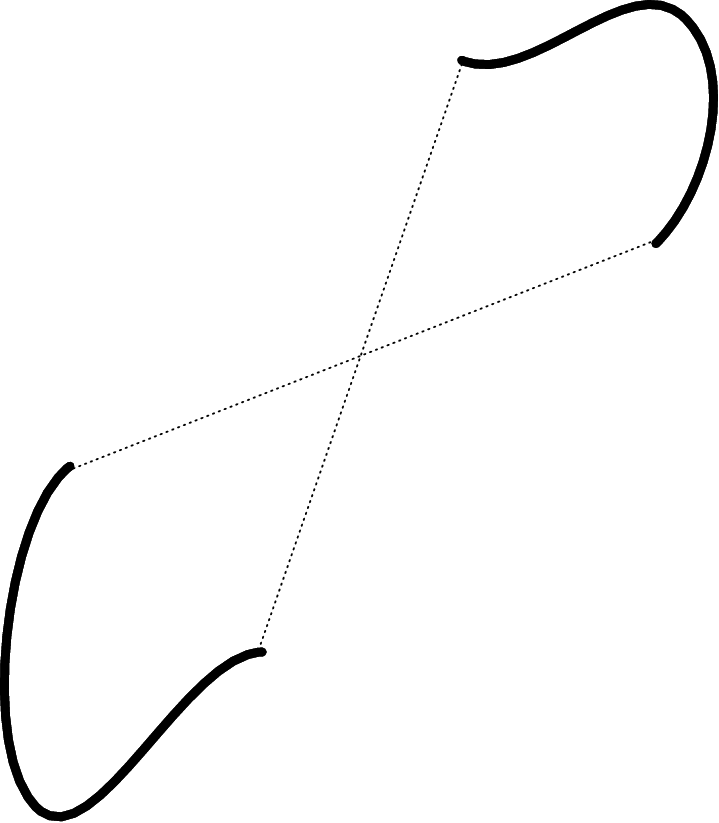}
	\caption{The action $\cA(\tau_Q)$ corresponding to the inscribed rectangle $Q$ is the area of the  (positive) region shown.}
	\label{fig:doubleicecream}
\end{figure}

\section{Spectral Invariants.}
\label{sec:spectral_invariants}
A useful reference for the general discussion in this section is provided by \cite{leczap}.

We are interested in studying the \emph{spectral invariants} of the group $\JF(\gamma, \theta)$ for pairs consisting of a real analytic Jordan curve $\gamma$ and aspect angle $0 < \theta < \pi$.  Spectral invariants are, in our case, real numbers associated to a choice of Jordan Floer homology class.  They arise since the Jordan Floer differential (and the Floer differential in general) respects the action filtration, allowing one to associate to a non-zero homology class the largest action among cycle representatives of that class.

More precisely, the action $\cA_{\theta H_t + h_t}$ gives a map
\[ \cA_{\theta H_t + h_t} \colon \cG(\gamma, \theta, h_t) \longrightarrow \bR \]
on the set of generators of the chain group of an admissible triple $(\gamma, \theta, h_t)$.  We then filter this chain group $\JFC(\gamma, \theta, h_t)$ by subgroups $\JFC^a(\gamma, \theta, h_t)$ generated by all elements of $\cG(\gamma, \theta, h_t)$ whose action is less than $a$.
Choosing an admissible $J_t \in \jreg(\gamma, \theta, h_t)$ gives a chain map $\partial = \partial_{(\gamma, \theta, h_t, J_t)}$ and this respects the filtration.

Then each non-zero homology class $g \in \JF(\gamma, \theta, h_t, J_t)$ has a filtration grading in the induced filtration on $\JF$, and this is called the spectral invariant of $g$:
\[ \ell(\gamma, \theta, h_t, J_t, g) \in \bR {\rm .} \]
Since we have continuation maps establishing independence of the chain homotopy class on almost-complex structures, and those continuation maps are filtered of filtration degree arbitrarily close to zero, we can remove the dependence of $\ell$ on $J_t$, so we obtain $\ell(\gamma, \theta, h_t, g)$.  Furthermore, spectral invariants satisfy a continuity property elaborated below (see \emph{continuity} result of \cite[Theorem 36]{leczap}).  Taking $h_t \to 0$,
we arrive at a spectral invariant
\[  \ell(\gamma, \theta, g) \in \bR {\rm ,} \]
which is the action of some trajectory $\tau$ of the Hamiltonian vector field $X_{\theta H_t}$ (see \emph{spectrality} result of  \cite[Theorem 36]{leczap}).  Such trajectories correspond to inscribed rectangles $Q_\tau$ in $\gamma$ of aspect angle $\theta$.

Spectral invariants satisfy several important properties, and they will be used to establish \Cref{thm:chunky}.
Recall that $\JF(\gamma, \theta) = \langle \alpha, \beta \rangle$, where $\alpha, \beta$ are homogeneous of Maslov degree $\vert \alpha \vert = 1$, $\vert \beta \vert = 2$.
We note without proof that the spectral invariants $\ell(\gamma, \theta, \alpha)$ and $\ell(\gamma, \theta, \beta)$ satisfy a duality property:
\[ \ell(\gamma, \theta, \alpha) + \ell(\gamma, \pi - \theta, \beta) = \area(\gamma) {\rm .} \]
We shall not need to invoke this duality in what follows, but it justifies restricting our attention to
\[ \ell(\gamma, \theta) := \ell(\gamma, \theta, \beta) \]
in order to simplify notation in the rest of this paper.
The duality property is proven in the sequel \cite{greenelobb4}.
Extend the domain of the spectral invariant to $[0,\pi]$ by setting $\ell(\gamma,0) = 0$ and $\ell(\gamma,\pi) = \area(\gamma)$.

\begin{prop}
	[Properties of spectral invariants]
	\label{prop:properties_of_spectral_invariants}
	Let $\gamma$ be a real analytic Jordan curve.
	Then the spectral invariant
	\[ \ell(\gamma, \cdot) \colon [0,\pi] \longrightarrow \bR \colon \theta \longmapsto \ell(\gamma, \theta) \]
	satisfies the following properties:
	\begin{itemize}
		\item \emph{Spectrality.}  The value $\ell(\gamma, \theta)$ is equal to the action of an inscribed $\theta$-rectangle of $\gamma$.
		\item \emph{Monotonicity.}  The function $\ell(\gamma,\cdot)\colon [0,\pi] \longrightarrow \bR$ is monotonic increasing.
		\item \emph{Area bound.} $0 \leq \ell(\gamma, \theta) \leq \area(\gamma)$.
		\item \emph{Lipschitz bound.}  The function $\ell(\gamma,\cdot) \colon [0,\pi] \longrightarrow \bR$ is Lipschitz continuous with Lipschitz constant $\rad(\gamma)^2$.
	\end{itemize}
\end{prop}

\begin{proof}
These properties follow from corresponding properties of the spectral invariants collected in \cite[Theorem 3]{leczap}, which carry over directly to our setting without change:
\begin{itemize}
\item
Spectrality is immediate.
\item
Monotonicity on $(0,\pi)$ follows from monotonicity in \cite[Theorem 3]{leczap}, taking $H = \theta_1 \cdot |z-w|^2/4$ and $K = \theta_2 \cdot |z-w|^2/4$ with $0 < \theta_1 < \theta_2 < \pi$.
\item
The area bound can be inferred from \cite[Section 4]{leczap}.
More concretely, for $\theta$ sufficiently close to 0 or $\pi$, a $\theta$-rectangle inscribed in $\gamma$ is elegant and contained in a small neighborhood of a binormal of $\gamma$.
This implies that it will have double ice-cream area in $(0,\area(\gamma))$ and very close to 0 or to $\area(\gamma)$, respectively.
The area bound then follows for all $\theta$ by monotonicity on $(0,\pi)$; and monotonicity on all of $[0,\pi]$ follows in turn.
\item
The Lipschitz bound follows by combining continuity and Lagrangian control in \cite[Theorem 3]{leczap}, noting that, by definition, $\max \{ \theta \cdot |z-w|^2/4 : (z,w) \in \gamma \times \gamma \} = \theta \cdot \rad(\gamma)^2$.
\end{itemize}
\end{proof}

The following corollary follows immediately:

\begin{cor}
	\label{cor:spectral_interval}
	Let $\gamma$ be a real analytic Jordan curve, and let $0 < \epsilon < \area(\gamma)/2$.
	Then there exists an interval $I \subset (0,\pi)$ of width $|I| \ge (\area(\gamma)-2\epsilon)/\rad(\gamma)^2$ such that for every $\theta \in I$, there exists a $\theta$-rectangle $Q \subset \gamma$ of associated action $\cA(\tau_Q)$ satisfying
	\[
	\epsilon < \cA(\tau_Q) < \area(\gamma)-\epsilon. 
	\]
\end{cor}

\begin{proof}
Take $a = \inf \{ \theta: \ell(\gamma,\theta) \ge \epsilon \}$, $b = \sup \{ \theta: \ell(\gamma,\theta) \le \area(\gamma)-\epsilon\}$, and $I = (a,b)$.
Certainly $a$ and $b$ are well-defined, since $\ell(\gamma,\pi) = \area(\gamma) > \epsilon$ and $\ell(\gamma,0) = 0 < \epsilon$.
By monotonicity, $\epsilon < \ell(\gamma,\theta) < \area(\gamma)-\epsilon$ for all $\theta \in I$.
By the Lipschitz bound, $\area(\gamma)-2\epsilon = \ell(\gamma,b) - \ell(\gamma,a) \le \rad(\gamma)^2 \cdot (b-a)$, so $|I| \ge (\area(\gamma)-2\epsilon)/\rad(\gamma)^2$.
The result now follows by spectrality.
\end{proof}

\section{Inscriptions.}
\label{sec:inscription_results}
We now turn to the question of how bounds on spectral invariants may be used to show that rectifiable curves inscribe non-degenerate $\theta$-rectangles.

The core of the argument is the following lemma:

\begin{lem}
[No shrinkout]
	\label{lem:approximating_rectifiable}
	Suppose that $\gamma_n$ is a sequence of parametrized real analytic Jordan curves of bounded length, which converge in $C^0$ to a Jordan curve $\gamma$, with $\area(\gamma_n) = \area(\gamma)$ for all $n$.
	Suppose further that each $\gamma_n$ inscribes a non-degenerate $\theta$-rectangle $Q_n \subset \gamma_n$ of associated action $\cA(\tau_{Q_n})$ satisfying
\[
\epsilon < \cA(\tau_{Q_n}) < \area(\gamma) - \epsilon
\]
for some $\epsilon > 0$.
Then $\gamma$ inscribes a non-degenerate $\theta$-rectangle $Q$.
\end{lem}

If we omit the hypothesis that the $\gamma_n$ have bounded length, and $\gamma$ does {\em not} inscribe a non-degenerate $\theta$-rectangle, then the proof of \Cref{lem:approximating_rectifiable} shows that every convergent subsequence of the $Q_n$ shrinks out to a point of non-rectifiability of $\gamma$.

\begin{proof}
	First, by passing to a subsequence if necessary, we use compactness to find a (possibly degenerate) rectangle $Q$ such that $Q_n \rightarrow Q$.  Assume for a contradiction that $Q$ is degenerate.  The contradiction we derive shall take the form of showing that $\cA(\tau_{Q_n}) \rightarrow 0 \in \bR/\area(\gamma)\bZ$ as $n \rightarrow \infty$.
	
	Each $\gamma_n \setminus Q_n$ consists of four components $\gamma_n^1, \gamma_n^2, \gamma_n^3, \gamma_n^4$.  Since $\gamma$ is a Jordan curve (and by possibly reordering) we may assume that $\gamma_n^2, \gamma_n^3, \gamma_n^4 \rightarrow Q$ in the Hausdorff metric
	while $\gamma_n^1 \rightarrow \gamma$.  For convenience, we choose shrinking discs $B_n \subset \bC$, centred at $Q$, such that $B_n \rightarrow Q$ and $\gamma_n \setminus \gamma_n^1 \subset B_n$.
	Now we consider the action \[
	\cA_{\theta H_t}(\tau_{Q_n}) = \int_0^1 \theta  H_t \circ \tau_{Q_n}(t) dt - \int_{\widehat{\tau_{Q_n}}} \omega {\rm .}
	\]
	There are two terms to this action.  The first term satisfies
	\[ \int_0^1 \theta H_t \circ \tau_{Q_n}(t) dt = \theta \rad(Q_n)^2 \longrightarrow 0.\]
	
	To compute the second term we would like to find an admissible cap $\widehat{\tau_{Q_n}}$ for the trajectory $\tau_{Q_n}$.  In particular we are interested in the boundary of this cap which will be a loop $\tau_{Q_n} \cup P_n$ where $P_n$ is a path on $\gamma_n \times \gamma_n$ connecting the two endpoints of $\tau_{Q_n}$.
	
	We start by choosing a path $P_n$ on $(\gamma_n \times \gamma_n) \setminus \Delta(\gamma_n)$ such that $\pi_1 \circ P_n$ and $\pi_2 \circ P_n$ are both injective paths on $\gamma_n$.  Note that we can ensure that we avoid the diagonal $\Delta(\gamma_n)$ by choosing $P_n$ so that $\pi_1 \circ P_n$ and $\pi_2 \circ P_n$ travel in the same direction around $\gamma_n$.
	
	Our first candidate for the boundary of $\widehat{\tau_{Q_n}}$ is $\tau_{Q_n} \cup P_n$.  The problem is that this cap may have non-zero winding number around the diagonal $\Delta(\bC)$.  To rectify this, we change $P_n$ by winding a number of times around the core curve of $\gamma_n \times \gamma_n \setminus \Delta({\gamma_n})$, which has winding number $1$ around the diagonal $\Delta(\bC)$.  This results in the curve $P_n'$.
	
	Now, computing inside $\bR/\area(\gamma)\bZ$, we have
	\begin{align*}
	\int_{\widehat{\tau_{Q_n}}} \omega &= \int_{\tau_{Q_n} \cup P_n'} \eta = \int_{\tau_{R_n} \cup P_n} \eta
	\end{align*}
	(where $d \eta = \omega$)
	since a core curve of $\gamma_n \times \gamma_n \setminus \Delta(\gamma_n)$ bounds area $2 \area(\gamma)$.
	
	On the other hand, we have
	\[ \int_{\tau_{Q_n} \cup P_n} \eta =  \int_{\pi_1(\tau_{Q_n} \cup P_n)} xdy + \int_{\pi_2(\tau_{Q_n} \cup P_n)} xdy\]
	which is nothing more than the sum of the areas of the two regions bounded by the (possibly self-intersecting) closed curves $\pi_1(\tau_{Q_n} \cup P_n)$ and $\pi_2(\tau_{Q_n} \cup P_n)$.
	
	Now for $i = 1,2$, $\pi_i(\tau_{Q_n} \cup P_n)$ is a curve that is the union of a small circular arc contained in $B_n$, and the closures of a subset of the arcs $\gamma_n^1, \gamma_n^2, \gamma_n^3, \gamma_n^4$.  At the expense of possibly changing the area by $\pm \area(\gamma)$ by adding on $\pm \gamma_n$, this is the area enclosed by a curve of length bounded by $2L+1$ (for large enough $n$) entirely contained in $B_n$, where $L$ denotes a uniform upper bound on the length of the $\gamma_n$.  But the area enclosed by such a curve contained in $B_n$ can be at most
	\[ \frac{2L+1}{{\rm Circumference}(B_n)} {\rm Area}(B_n) \]
	and this tends to $0$ as $n \rightarrow \infty$.
	
	Thus we see that as $n \rightarrow \infty$, the second term of the action tends to $0 \in \bR / \area(\gamma) \bZ$, and so we have arrived at a contradiction.
\end{proof}

\begin{proof}
[Proof of \Cref{thm:chunky}]
Suppose that $\gamma$ is a rectifiable Jordan curve.
The Riesz-Privalov theorem implies that $\gamma$ can be approximated in $C^0$ by a sequence of real analytic Jordan curves $\gamma_n$ of bounded length \cite[Theorem 6.8]{pommerenke}.
By rescaling these curves, we may assume additionally that $\area(\gamma_n) = \area(\gamma)$ for all $n$.
Now apply \Cref{cor:spectral_interval} and \Cref{lem:approximating_rectifiable}, taking $\epsilon \to 0$, in order to obtain the desired result for $\gamma$.
\end{proof}

\newpage

\bibliographystyle{amsplain}
\bibliography{References./works-cited.bib}

\providecommand{\bysame}{\leavevmode\hbox to3em{\hrulefill}\thinspace}
\providecommand{\MR}{\relax\ifhmode\unskip\space\fi MR }
\providecommand{\MRhref}[2]{%
  \href{http://www.ams.org/mathscinet-getitem?mr=#1}{#2}
}
\providecommand{\href}[2]{#2}
\begin{thebibliography}{10}

\bibitem{audindamian}
Mich\`ele Audin and Mihai Damian, \emph{Morse theory and {F}loer homology},
  Universitext, Springer, London; EDP Sciences, Les Ulis, 2014, Translated from
  the 2010 French original by Reinie Ern\'{e}.

\bibitem{aurouxintro}
Denis Auroux, \emph{A beginner's introduction to {F}ukaya categories}, Contact
  and symplectic topology, Bolyai Soc. Math. Stud., vol.~26, J\'{a}nos Bolyai
  Math. Soc., Budapest, 2014, pp.~85--136.

\bibitem{floerlag}
Andreas Floer, \emph{Morse theory for {L}agrangian intersections}, J.
  Differential Geom. \textbf{28} (1988), no.~3, 513--547.

\bibitem{floer_witten}
Andreas Floer, \emph{Witten's complex and infinite-dimensional {M}orse theory},
  J. Differential Geom. \textbf{30} (1989), no.~1, 207--221.

\bibitem{ganatrapormerleano}
Sheel Ganatra and Daniel Pormeleano, \emph{A log {P}{S}{S} morphism with
  applications to {L}agrangian embeddings}, J. {T}opol. \textbf{14} (2021),
  291--368.

\bibitem{greenelobb2}
Joshua~Evan Greene and Andrew Lobb, \emph{Cyclic quadrilaterals and smooth
  {J}ordan curves}, Invent. Math. \textbf{234} (2023), no.~3, 931--935.

\bibitem{greenelobb4}
\bysame, \emph{Square pegs between two graphs}, {\tt arxiv.org/2407.07798}
  (2024).

\bibitem{leczap}
R\'{e}mi Leclercq and Frol Zapolsky, \emph{Spectral invariants for monotone
  {L}agrangians}, J. Topol. Anal. \textbf{10} (2018), no.~3, 627--700.

\bibitem{matschke2014}
Benjamin Matschke, \emph{A survey on the square peg problem}, Notices Amer.
  Math. Soc. \textbf{61} (2014), no.~4, 346--352.

\bibitem{slimmcduffsalamon}
Dusa McDuff and Dietmar Salamon, \emph{{$J$}-holomorphic curves and quantum
  cohomology}, American Mathematical Society, Providence, RI, 1994.

\bibitem{oh1996}
Yong-Geun Oh, \emph{Floer cohomology, spectral sequences, and the {M}aslov
  class of {L}agrangian embeddings}, Internat. Math. Res. Notices (1996),
  no.~7, 305--346.

\bibitem{polto1991}
Leonid Polterovich, \emph{The {M}aslov class of the {L}agrange surfaces and
  {G}romov's pseudo-holomorphic curves}, Trans. Amer. Math. Soc. \textbf{325}
  (1991), no.~1, 241--248.

\bibitem{pommerenke}
Christian Pommerenke, \emph{Boundary behavior of conformal maps},
  Springer-Verlag, Berlin, 1992.

\bibitem{schmaschke}
Felix Schm{\"a}schke, \emph{Abelianization and {F}loer homology of
  {L}agrangians in clean intersection}, Ph.D. thesis, Universit{\"a}t
  {L}eipzig, 2017.

\bibitem{schnirelman1929}
Lev Schnirelmann, \emph{On certain geometrical properties of closed curves},
  Uspehi Matem. Nauk \textbf{10} (1944), 34--44.

\bibitem{toeplitz1911}
Otto Toeplitz, \emph{Ueber einige {A}ufgaben der {A}nalysis situs},
  Verhandlungen der {S}chweizerischen {N}aturforschenden {G}esellschaft (1911),
  no.~4, 197.

\bibitem{viterbo1990}
Claude Viterbo, \emph{A new obstruction to embedding {L}agrangian tori},
  Invent. Math. \textbf{100} (1990), no.~2, 301--320.

\end{thebibliography}
\end{document}